\documentclass[11pt]{article}
\pdfoutput=1

\usepackage[T1]{fontenc}
\usepackage[utf8]{inputenc}
\usepackage{lmodern}
\usepackage{microtype}

\usepackage{amsmath,amssymb,amsthm,latexsym,mathtools}
\usepackage{float}
\usepackage{booktabs}
\usepackage{multirow}
\usepackage{graphicx}
\usepackage{xcolor}
\usepackage{booktabs}
\usepackage{nicematrix}
\usepackage[ruled,vlined]{algorithm2e}
\usepackage{enumitem}
\usepackage[left=1in,right=1in,top=0.8in,bottom=0.9in]{geometry}
\usepackage{changepage}
\usepackage{hyperref}
\usepackage{cleveref}
\usepackage{float}
\usepackage{changepage} 

\usepackage[numbers,sort&compress]{natbib}

\newcommand{\hl}[1]{#1}
\newcommand{\highlighting}[1]{#1}

\newcommand{\Prob}{\ensuremath{\mathbb{P}}}
\newcommand{\Ex}{\ensuremath{\mathbb{E}}}

\def\maxim{\mathop{\textup{maximize}}}
\def\minim{\mathop{\textup{minimize}}}

\theoremstyle{plain}
\newtheorem{theorem}{Theorem}
\newtheorem{lemma}{Lemma}
\newtheorem{proposition}{Proposition}

\theoremstyle{definition}

\newtheorem{assumption}{Assumption}

\title{Age of information cost minimization with no buffers, random 
arrivals and unreliable channels: A PCL-indexability analysis}
\author{Jos\'e Ni\~no-Mora\\
Departamento de Estad\'{\i}stica, Universidad Carlos III de Madrid\\
28903 Getafe (Madrid), Spain\\
\texttt{jose.nino@uc3m.es}}
\date{}

\begin{document}

\maketitle

\begin{abstract}
Over the last decade, the Age of Information has emerged as a key concept and 
metric for applications where the freshness of sensor-provided data is critical. Limited 
transmission capacity has motivated research on the design of tractable policies for 
scheduling information updates to minimize Age of Information cost based on Markov 
decision  models, in particular on the restless multi-armed bandit problem (RMABP). This 
allows {the} use of Whittle’s popular index policy, which is often nearly optimal, provided 
indexability (index existence) is  proven, which has been recently accomplished in some 
models. We aim to extend the application scope of  Whittle’s index policy in a broader AoI 
scheduling  model. We address a  model with no buffers incorporating random packet 
arrivals, unreliable channels, and nondecreasing AoI costs. We use sufficient indexability conditions based on partial conservation laws 
previously introduced by the author to establish the model's indexability and evaluate its 
Whittle index in closed form  under discounted and average cost 
criteria. We further use the index formulae to draw insights on how 
scheduling priority depends on model parameters.
\end{abstract}



\medskip
\noindent\textbf{Note:} Published in \emph{Mathematics} \textbf{11}, 4394 (2023). DOI:\ \url{https://doi.org/10.3390/math11204394}

\section{Introduction}
\label{s:intro}
The last decade has witnessed a fast-growing interest in research on {the} 
 \emph{\hl{Age of Information}} 
 (AoI) concept and metric {as} proposed in
Kaul~et~al.~\cite{kauletal11,kauletal12} and extensions thereof.  {The AoI} 
 quantifies the freshness of information at a receiver in a communication system, as
it measures the time elapsed since the most recently received packet was generated at its source.
See, e.g., Kosta~et~al.~\cite{kostaetal17} and Yates~et~al.~\cite{yatesetal21} for authoritative introductions and surveys.
{AoI metrics} have become increasingly relevant as the demand for real-time applications, where timely information delivery is critical, continues to grow.  {These include} Internet of Things, video streaming, online gaming, remote control, and real-time traffic monitoring. 

\subsection{AoI Minimization via Transmission Scheduling and Restless Bandit Framework}
\label{s:aomrmabpf}
{Much AoI research} has {addressed the design of  transmission scheduling policies} to 
minimize AoI costs in {communication networks} {so that users have access to  as much 
up-to-date information as possible}. 
{A relatively simple yet relevant setting for addressing the optimal scheduling problem is that of} single-hop wireless communication networks, where information \emph{\hl{sources}} generate packets targeted to corresponding \emph{\hl{users}}.
Packet transmissions are mediated by a \emph{\hl{base station}} (BS),  
{i.e., a communications node that connects sources to users through a given set of channels}.
{In such a scenario,} the scheduling problem  can be {formulated} as {a restless multi-armed bandit problem (RMABP)}.  
See, e.g., the work reviewed in Section~\ref{s:rpw}.

The  RMABP is a versatile \emph{Markov decision process} (MDP) (see, e.g., Puterman~\cite{put94} and Bertsekas~\cite{bertsek12}) modeling framework  for optimal dynamic priority allocation, introduced by Whittle in~\cite{whit88b} as an extension of the classic \emph{multi-armed bandit problem} (MABP). 
{It} concerns the optimal dynamic selection  of up to $M$ out of $N$ \emph{projects}, {with $1 \leqslant M < N$. Such projects are} generic entities modeled as binary-action MDPs, {which} can be either active/{selected} or passive/rested at each time.
In the present setting, there are $N$ \emph{projects} representing \emph{users} {and $M$ channels}.
{Selecting a project means allocating a channel for attempting to transmit a packet held at 
the BS to the corresponding user}. 

An intuitively appealing and tractable type of policies for the RMABP is provided by \emph{priority-index policies} (or \emph{index policies} for short).
These {attach} an \emph{index} to each project/user, {i.e., a scalar function of its state, 
which} is used as a dynamic measure of selection priority: the larger a project's current 
index value, the higher the selection priority.

Index policies can be optimal in special models, in particular in the classic MABP (where 
passive projects do not change state) with $M = 1$, as first shown by Gittins and 
Jones~\cite{gijo74}, and in Klimov's model~\cite{kl74}. 
Yet, {they are generally suboptimal for} the RMABP, which is {computationally} intractable 
(P-space hard, see Papadimitriou and Tsitsiklis~\cite{papTsik99}).
Still,
Whittle~\cite{whit88b} {proposed a now widely popular heuristic} index policy for  the 
RMABP, {which he conjectured to be asymptotically optimal} under the average criterion, 
as {$M$ and $N$ grow to infinity in a fixed ratio.}
{Weber and Weiss~\cite{wewe90} showed that such a conjecture does not hold in general 
but identified sufficient conditions under which it does hold.}

{Since} its introduction,  research on the RMABP has generated a vast {amount of} 
literature, which was partially reviewed in~\cite{nmmath23}.
{Much work has addressed} applications of Whittle's index policy in a wide variety of {models}, where numerical studies often report a {near-optimal} performance. 
Yet, researchers aiming to apply such a policy to a particular RMABP model are confronted with 
two major roadblocks, widely regarded as difficult: (1) proving that the model is 
\emph{indexable}, {meaning} that it possesses a well-defined Whittle index, which cannot 
be taken for granted, and (2) devising an efficient index {evaluation} scheme. 

Over the last two decades, the author has developed a 
{framework} based on \emph{partial conservation laws} (PCLs) to {overcome such 
roadblocks} in increasingly general settings. See, e.g.,  
Ni\~{n}o{-}Mora~\cite{nmaap01,nmmp02,nmmor06,nmtop07,nmmor20}.
This approach is different from prevailing approaches and still not widely applied.
{We will apply it here, demonstrating} its simplicity and effectiveness for solving a 
hitherto unsolved problem: {proving} indexability and evaluating the Whittle index for an 
AoI optimization model in which such results had not been previously obtained.

\subsection{Related Prior Work: AoI Cost Minimization via Whittle's Index Policy}
\label{s:rpw}
{A number of recent papers have
applied} the RMABP and Whittle's index policy for AoI minimization in  single-hop wireless communication network models.
A {major stream of  such} work considers no-buffer models where  untransmitted  packets in a time slot are {lost}. 
Kadota~et~al.~\cite{kadotaetal18} {addressed a no-buffer network model having 
unreliable channels with a fixed transmission success probability for each user}. 
Packet generation was deterministic:  {each source generated one packet at the start of each 
frame}. 
The frames consisted of a fixed number of time slots {at which  transmissions could be 
attempted}. 
{If a packet transmission failed, it could be reattempted in the following slot} of the current 
frame, if any.
The goal was to minimize the (long-run) average linearly weighted AoI cost. 
Indexability was proven, the Whittle index was obtained in closed form, and, in a 
simulation study, Whittle's policy was benchmarked and shown to be nearly optimal and 
to outperform alternative index policies.

Hsu~et~al.~\cite{hsuetal20} considered an RMABP model for average linearly weighted 
AoI minimization  {incorporating random packet arrivals transmitted through reliable 
channels}.  
Indexability was established, Whittle's index was derived in closed form, and the resulting 
policy was tested. The \emph{online} setting, where arrival rates {are} inferred, was also 
addressed.

Sun~et~al.~\cite{sunetal17}  and Jhunjhunwala and Moharir~\cite{jhunjMoharir18} 
proposed the adoption of nonlinear AoI cost metrics to  {better model the growth of user 
dissatisfaction with data staleness}.
Tripathi and Modiano~\cite{tripMod19} incorporated {average cost metrics into an 
RMABP model with nondecreasing AoI costs in a no-buffer system with  deterministic 
arrivals and unreliable channels.
They proved indexability, obtained Whittle's index in closed form, and tested Whittle's 
index policy}. 
 
{Note that the aforementioned work does not address models incorporating both random packet arrivals and unreliable channels, which is a gap that the present paper fills.}

A related stream of work {focuses on establishing the} asymptotic optimality of Whittle's index policy in the above models under {appropriate} assumptions. 
Maatouk~et~al.~\cite{maatouketal21} established {asymptotic optimality for the model 
in~\cite{kadotaetal18} under a} numerically verifiable recurrence condition. 
Kriouile~et~al.~\cite{kriouileetal22} {proved such a result} under a less stringent 
condition.  

The no-buffer assumption is replaced in another stream of work by that of one-packet 
buffers, {where} the BS stores the last packet generated by a source.
Sun~et~al.~\cite{sunetal20} considered a {one-packet buffer variant of the model 
in~\cite{hsuetal20},} which was shown to result in performance gains. 
In addition to obtaining and testing Whittle's index policy, a decentralized index policy 
was proposed.
Sun~et~al.~\cite{sunetal19}  {extended the model in~\cite{sunetal20} by incorporating 
both random packet arrivals and unreliable channels}. 
Yet, {\cite{sunetal19} did not 
prove} that the model is indexable,
as it stated that ``\emph{it is really hard to establish indexability}''.
An index was {proposed} that is meant to approximate the 
Whittle index as if the model were indexable. 
 The {resulting} policy was benchmarked in a numerical study, showing that it is 
 near-optimal. 
Tang~et~al.~\cite{tangetal22}  extended the model in ~\cite{sunetal20} to incorporate  
{nondecreasing AoI costs}. They argued that the model is indexable, derived the Whittle 
index, and tested it in a numerical study.

The aforementioned work {has focused on the average  AoI cost minimization criterion}.  
Yet, recent work has {argued the relevance of  
the much less applied discounted cost criterion for such problems}. 
Badia and Munari~\cite{badiaMunari22} motivated {its use based on applications, e.g.,} in 
underwater environments,  where 
the assumption {of unlimited operation} is unreasonable.
{In such settings,  the discounted criterion, which  is well known to model a geometrically 
distributed time horizon, is more appropriate}. {In particular, Ref. 
\cite{badiaMunari22} focused on the performance evaluation of discounted AoI metrics} 
in a model with deterministic arrivals and unreliable channels, assuming  {that 
transmissions were attempted with a fixed probability}. 

Wu~et~al.~\cite{wuetal23} considered a discounted AoI minimization model  with 
deterministic arrivals and reliable channels,  proving a 
 property of optimal transmission policies.
 However, they did not consider the design of tractable suboptimal policies with good 
 performance. 
 
Zhang~et~al.~\cite{zhangetal21} investigated the pricing of AoI updates in a discounted 
model.  

\subsection{Contributions over Prior Work}
\label{s:copw}
This paper presents the following contributions:
\begin{enumerate}
\item {We  prove indexability for a no-buffer AoI optimization model with random packet 
arrivals, unreliable channels, and general nondecreasing costs, giving closed formulae for 
its Whittle index,  under both discounted and average cost criteria.  
Table~\ref{table:iarmabpaoim} clarifies the gaps filled by this paper on indexability under 
the average cost criterion. Note that the ``approx.'' beside Ref.~\cite{sunetal19} indicates 
that a proxy of the Whittle index was considered there, as indexability was not proven. 
As for indexability under the discounted criterion, to the best of the author's knowledge, it 
has not been addressed in an AoI setting. 
We thus demonstrate the effectiveness of the \emph{PCL-indexability} approach (based 
on \emph{partial conservation laws} (PCLs), 
see~\cite{nmaap01,nmmp02,nmmor06,nmtop07,nmmor20}) to prove indexability and 
evaluate the Whittle index by analyzing an AoI model that has not yet yielded to the 
prevailing approach.}

\item {We provide simplified Whittle index formulae for relevant special cases, in 
particular linear, quadratic, and threshold-type AoI costs. The resulting expressions allow 
for a more efficient index evaluation than direct application of the general formulae.}

\item {We analyze the Whittle index formulae to elucidate how a user's transmission priority depends on model parameters} under Whittle's policy, in particular 
{on the probabilities of packet arrival and successful transmission}.  
The index is nonincreasing in the former; so, other things being equal, Whittle's policy 
gives higher transmission priority to users with lower arrival rates.
As for the  latter probability, the index increases with it {for} linear costs, thus prescribing higher priority to users for which transmission attempts are more likely to succeed.
{Yet, for nonlinear costs, the index is generally non-monotonic in that  probability}. 
\end{enumerate}

\begin{table}[H]
\caption{\hl{Indexability} 
 analyses in RMABP AoI optimization models under the average cost criterion.}
\label{table:iarmabpaoim}
\setlength{\tabcolsep}{2.05mm}
\centering 
\begin{tabular}{rccccc}
\toprule
\multicolumn{1}{r}{\textbf{Arrivals $\backslash$\ Channels}} &
 &  \multicolumn{2}{c}{\textbf{Reliable}}
 & \multicolumn{2}{c}{\textbf{Unreliable}} \\
\midrule
\multirow{3}{*}{\textbf{Deterministic} \vspace{-6pt}} & 
\multicolumn{1}{r}{Costs $\backslash$\ Buffers}
 &  \multicolumn{1}{c}{No}
 & \multicolumn{1}{c}{Yes}  &  \multicolumn{1}{c}{No}
 & \multicolumn{1}{c}{Yes} \\
\cmidrule{3-6} 
& \multicolumn{1}{r}{Linear} &~\cite{kadotaetal18} &~\cite{sunetal20}&~\cite{kadotaetal18}  &~\cite{sunetal19} (approx.) \\
 \cmidrule{3-6} 
&\multicolumn{1}{r}{Nonlinear } &~\cite{tripMod19} &~\cite{tangetal22} &~\cite{tripMod19} &~\cite{tangetal22} \\
\midrule
\multirow{3}{*}{\textbf{Random} \vspace{-6pt}}&\multicolumn{1}{r}{costs $\backslash$\ buffers}
 &  \multicolumn{1}{c}{No}
 & \multicolumn{1}{c}{Yes} 
&  \multicolumn{1}{c}{No}
 & \multicolumn{1}{c}{Yes} \\
 \cmidrule{3-6}
 &\multicolumn{1}{r}{Linear}  &~\cite{hsuetal20} &~\cite{sunetal20}  &  this paper & ~\cite{sunetal19}  (approx.) \\
  \cmidrule{3-6}
 &\multicolumn{1}{r}{Nonlinear} & this paper &~\cite{tangetal22} & this paper &~\cite{tangetal22} \\
 
\bottomrule
\end{tabular}
\end{table}

\subsection{Structure of the Paper}
\label{s:softp}
The remainder of the paper is organized as follows.
Section~\ref{s:smpf} describes the system model and formulates it {as an} RMABP.
Section~\ref{s:iwi} reviews the concepts of indexability and Whittle index in the present setting.
Section~\ref{s:dpidwi} reviews the PCL-indexability approach and applies it to prove indexability and evaluate the Whittle index under the discounted criterion.
Section~\ref{s:pcliaac} {presents a PCL-indexability analysis for the average cost criterion}.
Section~\ref{s:dwimp} {elucidates the} dependence of Whittle's index on model 
parameters and evaluates the index for special cost functions. 
Finally, Section~\ref{s:discuss} {concludes the paper}.
The paper includes two appendices, which contain some of the given proofs.

\section{System Model and RMABP Formulation}
\label{s:smpf}
We consider a system model where $N$ information \emph{sources} (e.g., monitoring stations) generate status update packets aimed to $N$ corresponding \emph{users} (e.g., applications or sensors).
Time is slotted and time slots are denoted by $t = 0, 1, 2, \ldots$.
Packet generation is modeled by independent Bernoulli processes: source $n = 1, \ldots, N$ generates a packet at the start of a time slot with probability $0 < \lambda_n \leqslant 1$.

A BS {connects sources to users}. 
Transmission capacity is limited, as there are $M < N$ channels, so the BS can attempt to 
transmit at most $M$ packets per slot among those that have been just generated, if any 
exist.  
The BS has no buffering capacity, so any untransmitted packets in a slot are lost.

Channels are unreliable: if the BS attempts to transmit a freshly generated packet from 
source $n$ to user $n$, the transmission {succeeds} with probability $0 < \mu_n \leqslant 
1$ independently across channels and slots. 
If {the packet is transmitted}, the BS receives an instantaneous acknowledgment so that it 
can keep track of the current AoI for each user.
Packet generation and transmission error processes are mutually independent.

The \emph{state} of user $n$ at the beginning of slot $t$ is {denoted by} $Y_n(t) \triangleq (B_n(t), X_n(t))$, where 
$B_n(t)$ equals $1$ if a packet has just been generated by source $n$ and is $0$ otherwise, and $X_n(t)$ is the AoI for user $n$.
Thus, $Y_n(t)$ lies in the state space $\mathcal{Y} \triangleq \{0, 1\} \times \mathbb{N}$, where $\mathbb{N}$ is the {set of natural numbers with elements $1, 2, 3, \ldots$}, as we assume that $X_n(t) \geqslant 1$, because it takes one time slot for a packet to reach the user.

Figure~\ref{fig:aoisystem1} represents a system with $M = 3$ channels and $N = 5$ users 
in a slot where two packets from sources $1$ and $3$ have been generated. Only two 
channels can be allocated, which is indicated with gray filling.
Figure~\ref{fig:aoisystem2} shows a snapshot of the same system with all channels 
allocated in a slot, as there are at least as many packets available as {channels}. Yet, the 
packet from source $5$ will be lost, as it {cannot} be transmitted in the slot. 

\begin{figure}[H]
    \centering
    \includegraphics{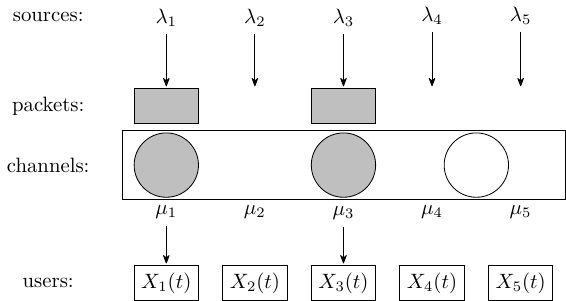}
    \caption{\hl{System} 
 snapshot showing underutilization of channels.}
    \label{fig:aoisystem1}
\end{figure}
\vspace{-12pt}
\begin{figure}[H]
    \centering
    \includegraphics{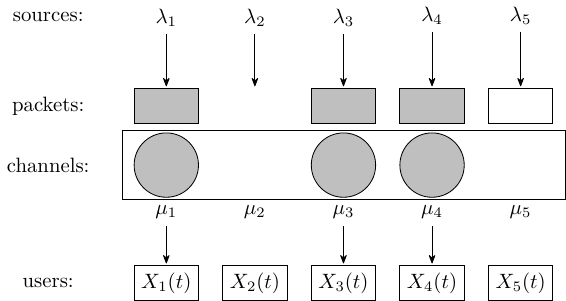}
    \caption{\hl{System} snapshot where a generated packet will be lost.}
    \label{fig:aoisystem2}
\end{figure}

The cost incurred by user $n$ {in} a slot is given by a general nonnegative and nondecreasing function $c_n(i_n)$ of its AoI $i_n$ at the beginning of the slot. 

At the start of each slot, {a controller at the BS} decides which packets to attempt to 
transmit, 
among those just generated, if any, up to a maximum of $M$.

We shall assume that 
{action choice is based only on the current joint  state.
This corresponds to considering} \emph{scheduling policies} $\pi$ from the class $\Pi(M)$ of  \emph{stationary policies} that attempt to transmit up to $M$ packets at each slot.
We denote by $A_n(t)$ the \emph{action} chosen for user $n$ in slot $t$, which can be $1$ 
(attempt transmission of an available packet) or $0$ otherwise (note that this action 
\emph{must} be taken if no packet is available).
As in~\cite{nmmp02}, we shall refer to $\mathcal{Y}^{\{0, 1\}} \triangleq \{1\} 
\times \mathbb{N}$ as the set of \emph{controllable states} for user $n$,
where both actions are available, {and to  $\mathcal{Y}^{\{0\}} \triangleq \{0\} \times  
\mathbb{N}$ as the set of 
\emph{uncontrollable states}, where only 
the passive action is available}. 

The state transition probabilities for user $n$ under action $a_n$, denoted by  $p_{(b_n, i_n), (b_n', i_n')}^{a_n} \triangleq \Prob\{Y_n(t+1) = (b_n', i_n') \, \vert \, Y_n(t) = (b_n, i_n), A_n(t) = a_n\}$ for any time slot $t$, are given by

\begin{equation}
\label{eq:prob0}
p_{(b_n, i_n), (b_n', i_n')}^{0}
 = 
\begin{cases}
1 - \lambda_n & \textup{ if } (b_n', i_n') = (0, i_n+1) \\
\lambda_n & \textup{ if } (b_n', i_n') = (1, i_n+1) \\
0 & \textup{otherwise,}
\end{cases}
\end{equation}
and, for $b_n = 1$,
\begin{equation}
\label{eq:prob1}
p_{(1, i_n), (b_n', i_n')}^{1} = 
\begin{cases}
(1 - \lambda_n) (1 - \mu_n) & \textup{ if } (b_n', i_n') = (0, i_n+1) \\
(1 - \lambda_n) \mu_n & \textup{ if } (b_n', i_n') = (0, 1) \\
\lambda_n (1 - \mu_n) & \textup{ if } (b_n', i_n') = (1, i_n+1) \\
\lambda_n \mu_n & \textup{ if } (b_n', i_n') = (1, 1) \\
0 & \textup{otherwise}.
\end{cases}
\end{equation}

Given a discount factor $0 < \beta < 1$, the expected total discounted cost incurred by policy $\pi \in \Pi(M)$ starting from \hl{the} 
 \emph{joint state} $\mathbf{Y}(0) = (\mathbf{B}(0), \mathbf{X}(0)) = (\mathbf{b}, \mathbf{i}) \in \{0, 1\}^N \times \mathbb{N}^N$, {where we use the vector notation
$\mathbf{B}(t) = (B_n(t))_{n=1}^N$, $\mathbf{X}(t) = (X_n(t))_{n=1}^N$, 
 $\mathbf{b} = (b_n)_{n=1}^N$ and $\mathbf{i} = (i_n)_{n=1}^N$}, is 
\begin{equation}
\label{eq:etdaoic}
{F_{(\mathbf{b}, \mathbf{i})}^\pi \triangleq} \, \Ex_{(\mathbf{b}, \mathbf{i})}^\pi\bigg[\sum_{t=0}^\infty \sum_{n=1}^N c_n(X_n(t)) \beta^t\bigg].
\end{equation}

{In~(\ref{eq:etdaoic}),} $\Ex_{(\mathbf{b}, \mathbf{i})}^\pi$ is the expectation under 
policy $\pi$, conditioned on starting from $(\mathbf{b}, \mathbf{i})$.
The discounted cost problem is to find an \emph{admissible} policy (i.e., in $\Pi(M)$) 
minimizing the cost objective in~(\ref{eq:etdaoic}) for \emph{any} initial joint state 
$(\mathbf{b}, \mathbf{i})$.

We shall also consider the average cost problem, which is to find an admissible scheduling 
policy minimizing the average expected cost per slot, given by
\begin{equation}
\label{eq:avaoic}
{\bar{F}_{(\mathbf{b}, \mathbf{i})}^\pi \triangleq}  \, \limsup_{T \to \infty} \, \frac{1}{T}
\Ex_{(\mathbf{b}, \mathbf{i})}^\pi\bigg[\sum_{t=0}^T \sum_{n=1}^N c_n(X_n(t))\bigg].
\end{equation}

{Note that we restrict attention to stationary policies, even though we have not proven that they are optimal among the wider class of history-dependent policies. 
Conditions ensuring such a result} in MDP models with countable states and unbounded costs are given, e.g., in Sections 6.10 (discounted criterion) and 8.10 (average criterion) of Puterman~\cite{put94}.
Yet, since such issues are only of ancillary interest to this paper, we do not pursue them here, except for certain single-user subproblems (see Proposition~\ref{pro:prop6105a}).

The above problems are {formulated as} RMABPs by identifying projects with users and 
taking the active action on a project with attempting to transmit a packet for a user. %

Given the intractability of computing an optimal policy, we aim to design suboptimal policies that can be implemented with low complexity and perform well.
\emph{Index policies} are appealing in this setting; they are based on defining an 
\emph{index} $\alpha_n\colon \mathcal{Y}^{\{0, 1\}} \to \mathbb{R}$ for each user 
$n$, a scalar function of its controllable states $(1, i_n)$.  
In a slot starting in $(\mathbf{b}, \mathbf{i})$, {where $\mathbf{b}$ and $\mathbf{i}$ are 
as in~(\ref{eq:etdaoic}),} {the policy attempts to transmit  up to $M$ available packets, if 
any exist, 
giving higher priority to users with} larger nonnegative index values and breaking ties 
arbitrarily.

This raises the issue of how to design easily computable index policies that perform well, if possible. 
Regarding ease of calculation, a simple choice is the \emph{greedy index} policy, based on the index 
$\alpha_n(1, i_n) \triangleq c_n(i_n)$.
A more sophisticated idea for constructing index policies was proposed by Whittle~\cite{whit88b}, as discussed below.

\section{Indexability and Whittle Index}
\label{s:iwi}
This section discusses indexability and the Whittle index, adapting the approach in  Whittle~\cite{whit88b} to the nuances of the present model.
{We focus} on the discounted cost problem, 
\begin{equation}
\label{eq:op}
\minim_{\pi \in \Pi(M)} \, F_{(\mathbf{b}, \mathbf{i})}^\pi,
\end{equation}
where {$F_{(\mathbf{b}, \mathbf{i})}^\pi$, defined in~(\ref{eq:etdaoic}),
is the system's \emph{cost metric}, giving the expected total $\beta$-discounted cost 
under policy $\pi$ starting from $(\mathbf{b}, \mathbf{i})$.}
We define $F_{(\mathbf{b}, \mathbf{i})}^*$ to be the optimal cost.  

We start by formulating a {\emph{relaxation} of optimization problem~(\ref{eq:op}) by
replacing the class of admissible 
scheduling policies $\Pi(M)$, which can allocate up to $K$ channels per slot, with a larger 
class of policies. 
In particular, we use the class of policies $\Pi(N)$, which can allocate up 
to $N$ channels per slot, 
and further impose the additional constraint

\begin{equation}
\label{eq:etdrvspc}
G_{(\mathbf{b}, \mathbf{i})}^\pi \leqslant M_\beta,
\end{equation}
where $M_\beta \triangleq M / (1-\beta)$ and 
\[
G_{(\mathbf{b}, \mathbf{i})}^\pi \triangleq \Ex_{(\mathbf{b}, \mathbf{i})}^\pi\Bigg[\sum_{t=0}^\infty \sum_{n=1}^N A_n(t) \beta^t\Bigg].
\]
{We shall refer to $G_{(\mathbf{b}, \mathbf{i})}^\pi$ as the system's \emph{work metric}, which gives} the expected total discounted number of transmission attempts. 

Note that~(\ref{eq:etdrvspc}) is implied by 
the sample-path constraints implicit in $\Pi(M)$, {namely,}
\[
\sum_{n=1}^N A_n(t) \leqslant M, \enspace t = 0, 1, \ldots
\]

The resulting  \emph{relaxed problem} is
\begin{equation}
\label{eq:rp}
\minim_{\pi \in \Pi(N),~(\ref{eq:etdrvspc})} \, F_{(\mathbf{b}, \mathbf{i})}^\pi.
\end{equation}
Note that, by construction, 
the optimal cost $F_{(\mathbf{b}, \mathbf{i})}^{\textup{R}}$ of~(\ref{eq:rp}) is a lower bound on $F_{(\mathbf{b}, \mathbf{i})}^*$. 

We now attach a nonnegative Lagrange multiplier, or, more properly, since we are dealing with an  inequality-constrained optimization problem, a Karush--Kuhn--Tucker (KKT) multiplier,
$\nu \geqslant 0$, to constraint~(\ref{eq:etdrvspc}) in~(\ref{eq:rp}), and formulate the \emph{Lagrangian function}
\begin{equation}
\label{eq:lagrfun}
\mathcal{L}_{(\mathbf{b}, \mathbf{i})}^\pi(\nu) \triangleq F_{(\mathbf{b}, \mathbf{i})}^\pi + \nu \big(G_{(\mathbf{b}, \mathbf{i})}^\pi- M_\beta\big),\end{equation}
and the 
\emph{Lagrangian relaxation}
\begin{equation}
\label{eq:lr}
 \minim_{\pi \in \Pi(N)} \, \mathcal{L}_{(\mathbf{b}, \mathbf{i})}^\pi(\nu).
 \end{equation}
 
Note that multiplier $\nu$ represents a \emph{charge per transmission attempt}.
Again, by construction, the optimal cost $\mathcal{L}_{(\mathbf{b}, \mathbf{i})}^*(\nu)$ of problem~(\ref{eq:lr}) gives in turn a lower bound on  $F_{(\mathbf{b}, \mathbf{i})}^{\textup{R}}$ for any $\nu \geqslant 0$, which follows from the elementary \emph{weak duality} relation 
\begin{equation}
\label{eq:wd}
\mathcal{L}_{(\mathbf{b}, \mathbf{i})}^*(\nu)  \leqslant F_{(\mathbf{b}, \mathbf{i})}^\pi, \enspace \textup{for any} \enspace \pi \in \Pi(N) \enspace \textup{satisfying} \enspace~(\ref{eq:etdrvspc}) \enspace \textup{and} \enspace
\nu \geqslant 0.
\end{equation}

The \emph{Lagrangian dual problem} is to find a multiplier yielding the best lower 
bound: 
\begin{equation}
\label{eq:ldp}
\maxim_{\nu \geqslant 0} \, \mathcal{L}_{(\mathbf{b}, \mathbf{i})}^*(\nu).
\end{equation}

Note that~(\ref{eq:ldp}) is a \emph{concave optimization problem}---since function $\mathcal{L}_{(\mathbf{b}, \mathbf{i})}^*(\nu)$ is concave in $\nu$, being a minimum of linear functions---which  facilitates its computational solution.

We would like to ensure \emph{strong duality}, meaning that 
$\mathcal{L}_{(\mathbf{b}, \mathbf{i})}^*(\nu^*)  = F_{(\mathbf{b}, \mathbf{i})}^{\pi^*}$ 
for some $\pi^* \in \Pi(N)$ satisfying~(\ref{eq:etdrvspc}) and $\nu^* \geqslant 0$, and 
hence, $F_{(\mathbf{b}, \mathbf{i})}^{\textup{R}} = F_{(\mathbf{b}, \mathbf{i})}^{\pi^*}$ 
and
$\mathcal{L}_{(\mathbf{b}, \mathbf{i})}^{**} = F_{(\mathbf{b}, \mathbf{i})}^{\textup{R}}$, where 
 $\mathcal{L}_{(\mathbf{b}, \mathbf{i})}^{**} = \mathcal{L}_{(\mathbf{b}, \mathbf{i})}^*(\nu^*)$ is the optimal  value of problem~(\ref{eq:ldp}).
For such a purpose, we adapt the concept of \emph{complementary slackness} (CS) from inequality-constrained optimization. 
Given a policy $\pi \in \Pi(N)$ and a multiplier $\nu \geqslant 0$, we say that they satisfy CS if
\begin{equation}
\label{eq:cscond}
\nu \big(G_{(\mathbf{b}, \mathbf{i})}^\pi - M_\beta\big) = 0.
\end{equation}

The next result gives a sufficient condition for strong duality, which can be used to compute the bound $F_{(\mathbf{b}, \mathbf{i})}^{\textup{R}}$.
\begin{proposition}
\label{pro:strongd}
Suppose that $\pi^* \in \Pi(N)$ solves {\rm~(\ref{eq:lr})} with $\nu^* \geqslant 0$, and 
that they satisfy CS. Then,
\begin{enumerate}[label=(\alph*)]
\item $\pi^*$ solves the relaxed problem {\rm~(\ref{eq:rp});}
\item $\nu^*$ solves the Lagrangian dual problem {\rm~(\ref{eq:ldp});}
\item Strong duality holds.
\end{enumerate}
\end{proposition}
\begin{proof}
The results follow straightforwardly from the above discussion and the  identities
\begin{align*}
F_{(\mathbf{b}, \mathbf{i})}^{\pi^*} - \mathcal{L}_{(\mathbf{b}, \mathbf{i})}^*(\nu^*)  & = F_{(\mathbf{b}, \mathbf{i})}^{\pi^*}  + \nu^* \big(G_{(\mathbf{b}, \mathbf{i})}^{\pi^*} - M_\beta\big) - \mathcal{L}_{(\mathbf{b}, \mathbf{i})}^*(\nu^*) = \mathcal{L}_{(\mathbf{b}, \mathbf{i})}^{\pi^*}(\nu^*) - \mathcal{L}_{(\mathbf{b}, \mathbf{i})}^*(\nu^*) = 0.
\end{align*}
\end{proof}

The solution of the Lagrangian relaxation~(\ref{eq:lr}) is further facilitated by the fact that 
it naturally decouples into the \emph{single-user subproblems}
\begin{equation}
\label{eq:sulr}
\minim_{\pi_n \in \Pi_n} \, F_{n, (b_n, i_n)}^{\pi_n} + \nu G_{n, (b_n, i_n)}^{\pi_n}, 
\end{equation}
for $n = 1, \ldots, N$, where $\Pi_n$ {denotes the stationary transmission policies for the subsystem corresponding to user $n$ \emph{in isolation} with a single channel. Note that}
\begin{equation}
\label{eq:costmn}
F_{n, (b_n, i_n)}^{\pi_n} \triangleq \Ex_{(b_n, i_n)}^{\pi_n}\Bigg[\sum_{t=0}^\infty c_n(X_n(t))  \beta^t\Bigg]
\end{equation}
and
\begin{equation}
\label{eq:workmn}
G_{n, (b_n, i_n)}^{\pi_n} \triangleq \Ex_{(b_n, i_n)}^{\pi_n}\Bigg[\sum_{t=0}^\infty A_n(t)  \beta^t\Bigg]
\end{equation} 
are the user's \emph{cost metric} and \emph{work metric}, respectively, as defined in Ni\~no-Mora~\cite{nmmp02}.
This follows from considering decoupled policies $\pi = (\pi_n)_{n=1}^N$, where $\pi_n$ is used on user $n$, so 
\[
F_{(\mathbf{b}, \mathbf{i})}^{\pi} = \sum_{n=1}^N F_{n, (b_n, i_n)}^{\pi_n} \quad \textup{and} \quad
G_{(\mathbf{b}, \mathbf{i})}^{\pi} = \sum_{n=1}^N G_{n, (b_n, i_n)}^{\pi_n},
\]
hence the Lagrangian function in~(\ref{eq:lagrfun}) can be decomposed as
\begin{equation}
\label{eq:lfdec}
\mathcal{L}_{(\mathbf{b}, \mathbf{i})}^\pi(\nu) = \sum_{n=1}^N \big(F_{n, (b_n, i_n)}^{\pi_n} + \nu G_{n, (b_n, i_n)}^{\pi_n}\big) - M_{\beta} \nu.
\end{equation}

We shall call the model \emph{indexable} if, for each user $n$ and AoI $j_n \in 
\mathbb{N}$, there exists a critical transmission attempt charge $\nu_{(1, j_n)}^*$, which, 
viewed as an index, characterizes the optimal policies for subproblem~(\ref{eq:sulr}) 
under any $\nu \in \mathbb{R}$, as follows: {in state $(1, j_n)$}, it is optimal---for any 
initial state $(b_n, i_n)$---to attempt to transmit  if and only if $\nu_{(1, j_n)}^* \geqslant 
\nu$, and it is optimal to not attempt to transmit if and only if $\nu_{(1, j_n)}^* \leqslant 
\nu$. We shall refer to $\nu_{(1, j_n)}^*$ as user $n$'s \emph{Whittle index}. 
Note that we define Whittle's indexability property as formulated in, e.g.,~\cite{nmasmta14, nmmor20}.

The \emph{Whittle index policy} {for the $N$-user model} gives higher transmission priority to available packets targeted to users with larger nonnegative indices.

\section{PCL-Indexability Analysis: Discounted Cost Criterion}
\label{s:dpidwi}
This section {reviews the PCL-indexability approach in~\cite{nmmor06,nmtop07} and 
applies it to prove indexability of the single-user subproblems~(\ref{eq:sulr}) and derive 
the discounted Whittle index}.

\subsection{A Verification Theorem for Threshold-Indexability}
\label{s:vtpictp}
Consider a single-user subproblem~(\ref{eq:sulr}), dropping the user label $n$ and writing the AoI cost function as $c_j$.
The discounted cost and work performance metrics  are {given by}
\begin{equation}
\label{eq:costm}
F_{(b, i)}^\pi \triangleq \Ex_{(b, i)}^\pi\Bigg[\sum_{t=0}^\infty c_{X(t)} \beta^t\Bigg],
\end{equation}
and 
\begin{equation}
\label{eq:workm}
G_{(b, i)}^\pi \triangleq \Ex_{(b, i)}^\pi\Bigg[\sum_{t=0}^\infty A(t) \beta^t\Bigg],
\end{equation}
so the \emph{single-user subproblem} of concern is {formulated as follows:}
\begin{equation}
\label{eq:ssulr}
\minim_{\pi \in \Pi} \, F_{(b, i)}^{\pi} + \nu G_{(b, i)}^{\pi}.
\end{equation}

We shall write the packet arrival and the transmission success probabilities as $\lambda$ and $\mu$, respectively, dropping the user label.
We shall further use the parameter $p \triangleq \lambda \mu$, the probability that in a 
time slot both a packet is  generated and a corresponding transmission attempt succeeds, 
and consider the complementary probability $q \triangleq 1 - p$. Note that $p > 0$, but 
$q$ may equal $0$, in the case of deterministic arrivals and a reliable channel ($\lambda = 
\mu = 1$).

We make the following standing assumption on the one-slot cost function $c_i$.

\begin{assumption}
\label{ass:ci}
\textup{ }
\begin{enumerate}[label=(\roman*)]
\item $c_i$ is nonnegative and nondecreasing$;$
\item $\displaystyle \sum_{i=1}^\infty  c_i q^i < \infty.$
\end{enumerate}
\end{assumption}

{Note that} part (ii) is a mild growth condition on $c_i$. 
{We next  show that Assumption~\ref{ass:ci} {implies} the existence of optimal stationary 
deterministic policies for subproblem~(\ref{eq:ssulr}). }

{
\begin{proposition}
 \label{pro:prop6105a}
 Single-user subproblem $(\ref{eq:ssulr})$ is solved optimally by stationary deterministic 
 policies determined by Bellman's equations.
 \end{proposition}}
 \begin{proof}
 The result is trivial in the case $q = 0$. In the case $q > 0$,
{we apply Theorem 6.10.4 in Puterman~\cite{put94}, which ensures the stated result 
provided that Assumptions 6.10.1 and 6.10.2 therein hold, which are formulated in 
terms of an appropriate weight function.  
Here, we take  $w_{(b, i)} \triangleq q^{-i}$ as the weight function, which satisfies 
Assumption 6.10.1 in~\cite{put94} because
 $c_i / w_{(b, i)} = c_i q^i$ is bounded under Assumption~\ref{ass:ci}.}
 
{To prove Assumption 6.10.2 in~\cite{put94}, we use the sufficient condition provided by
 Proposition 6.10.5.a in~\cite{put94}. In the present setting, this reduces to }
{\begin{equation}
\label{eq:prop6105a}
(1 - \lambda) (1 - \mu) w_{(0, i+1)} + 
(1 - \lambda) \mu w_{(0, 1)} + 
\lambda (1 - \mu) w_{(1, i+1)} + 
\lambda \mu w_{(1, 1)} \leqslant w_{(1, i)} + L,
\end{equation}}
{\hspace{-2pt}for some $L > 0$.
This follows by verifying that the expression
\begin{align*}
& (1 - \lambda) (1 - \mu) w_{(0, i+1)} + 
(1 - \lambda) \mu w_{(0, 1)} + 
\lambda (1 - \mu) w_{(1, i+1)} + 
\lambda \mu w_{(1, 1)} - w_{(1, i)} \\
& = (1 - \lambda) (1 - \mu) q^{-(i+1)} + 
(1 - \lambda) \mu q^{-1} + 
\lambda (1 - \mu) q^{-(i+1)} + 
\lambda \mu q^{-1} - q^{-i}  \\
& =  (1 - \mu) q^{-(i+1)} + 
 \mu q^{-1}  - q^{-i}
\end{align*}
is decreasing in $i$, and hence its maximum value is reached by $i = 0$. This allows us to take 
$L \triangleq \lambda \mu / (1 - \lambda \mu) = (1-q)/q > 0$ satisfying~(\ref{eq:prop6105a}).}
\end{proof}

We shall consider a class of structured policies that intuition suggests might be optimal for 
problem~(\ref{eq:ssulr}) for any transmission attempt charge $\nu \in \mathbb{R}$, and 
which will later be proven to be, indeed, optimal: \emph{threshold policies}, defined as 
follows.
For a nonnegative integer $k \geqslant 0$, the \emph{$k$-threshold policy} (or \emph{$k$-policy} for short) attempts to transmit a packet in controllable state $(1, j)$ if and only if $j > k$. 
We shall write its performance metrics as $F_{(b, i)}^k$ and $G_{(b, i)}^k$.

We shall further consider \emph{marginal performance metrics}, giving the increments in cost and work metrics under threshold policies resulting from a {change in the initial action}. Thus,  
the \emph{marginal cost metric} for the $k$-policy starting from $(1, i)$  is  
\begin{equation}
\label{eq:mcostm}
f_{(1, i)}^k \triangleq F_{(1, i)}^{\langle 0, k\rangle} - F_{(1, i)}^{\langle 1, k\rangle},
\end{equation}
where $\langle a, z\rangle$ denotes the policy that takes action $a$ in the initial time slot 
$t = 0$ and then follows the $k$-policy {from time $t = 1$ onward}.
Note that $f_{(1, i)}^k$ measures the {decrease in cost that results from modifying the $k$-policy by attempting to transmit {initially} in state  $(1, i)$ compared to not doing so}.
The \emph{marginal work metric} is 
\begin{equation}
\label{eq:mworkm}
g_{(1, i)}^k \triangleq G_{(1, i)}^{\langle 1, k\rangle} - G_{(1, i)}^{\langle 0, k\rangle},
\end{equation}
which measures the corresponding  increase in work expended. 

If $g_{(1, i)}^k > 0$ for every $i \in \mathbb{N}$, we further define the \emph{marginal 
productivity} (MP) \emph{metric}
\begin{equation}
\label{eq:mpm}
m_{(1, i)}^k \triangleq \frac{f_{(1, i)}^k}{g_{(1, i)}^k},
\end{equation}
and the \emph{MP index}
\begin{equation}
\label{eq:mpi}
m_{(1, i)} \triangleq m_{(1, i)}^i.
\end{equation}

Consider now the following \emph{PCL-indexability} (PCLI) conditions: 
\begin{itemize}
\item[]$($PCLI1$)$ Positive marginal work: $g_{(1, i)}^k > 0$ for {every} AoI $i \geqslant 1$ and threshold $k \geqslant 0$.
\item[]$($PCLI2$)$ Monotone nondecreasing MP index: $m_{(1, i)}$ is nondecreasing in $i$.
\end{itemize}

The following \emph{verification theorem} {formulates, in the present setting, a result} 
proven by the author in prior work for restless bandits in increasingly general settings. %
See  
 \cite[Corollary \hl{2}]{nmaap01} (finite-state projects), \cite[Theorem \hl{6.3} ]{nmmp02} (finite-state projects with general work metrics), \cite[Theorem \hl{4.1} ]{nmmor06} (semi-Markov countable-state projects), and  \cite[Theorem \hl{1} and Proposition \hl{1}(c)]{nmmor20} (real-state projects).
The result refers to \emph{threshold-indexability}---indexability consistent with threshold 
policies---meaning that \emph{both} the model is indexable \emph{and} threshold 
policies are optimal for subproblem~(\ref{eq:ssulr}) for any $\nu \in \mathbb{R}$. 

\begin{theorem}
\label{the:verif}
Suppose that condition $($PCLI1$)$ holds. Then, the model is threshold-indexable if 
and only if condition $($PCLI2$)$ holds, with Whittle index given by the MP 
index, i.e., $\nu_{(1, i)}^* = m_{(1, i)}$.
\end{theorem}

 Note that~\cite{nmaap01,nmmp02,nmmor06} show that  $($PCLI1$)$ and $($PCLI2$)$ are sufficient conditions for indexability of discrete-state projects, with the Whittle index being given by the MP index. 
{The characterization of threshold indexability in Theorem~\ref{the:verif} under $($PCLI1$)$ is 
stated in~\cite{nmmor20} as a consequence of Theorem 1 and Proposition 1(c) therein.}
 
We emphasize that the PCL-indexability approach \emph{does not entail proving  first 
optimality of threshold policies}, as in the currently prevailing approach to indexability:
\emph{both} indexability \emph{and} optimality of threshold policies follow in one fell 
swoop from Theorem~\ref{the:verif}.

{We next set out to  apply} Theorem~\ref{the:verif} to the present model, which entails verifying PCL-indexability conditions $($PCLI1$)$ and $($PCLI2$)$. 
\label{s:avtpaoim}
\subsection{Work Metric Analysis and Condition \highlighting{$($PCLI1$)$}} 
\label{s:wmacpcli1}
We start with condition $($PCLI1$)$, namely {that the marginal work metric under threshold policies be positive}. 
Recall that the $k$-policy, with threshold $k \geqslant 0$, takes the active action (attempt to transmit) in state $(1, i)$ if $i > k$, and the passive action otherwise.

{We next} address evaluation of the work metric $G_{(b, i)}^k$ under the $k$-policy. 
Even though in the system model we assumed that the AoI $i$ is a positive integer, for obtaining the $G_{(b, i)}^k$ we include the fictitious value $i = 0$.

From the dynamics in~(\ref{eq:prob0}) and~(\ref{eq:prob1}) it follows that the work metric $G_{(b, i)}^k$ satisfies the following linear equations: for $i, k \geqslant 0$,
\begin{equation}
\label{eq:G0ik}
G_{(0, i)}^k = \beta (1-\lambda) G_{(0, i+1)}^k + \beta \lambda G_{(1, i+1)}^k{,}
\end{equation}
and
\begin{equation}
\label{eq:G1ik}
G_{(1, i)}^k = 
\begin{cases}
\beta (1-\lambda) G_{(0, i+1)}^k + \beta \lambda G_{(1, i+1)}^k, & \enspace i \leqslant k \\
1 + \beta (1-\lambda) \mu G_{(0, 1)}^k + \beta  \lambda  \mu G_{(1, 1)}^k  & \\
\quad + \beta (1-\lambda) (1-\mu) G_{(0, i+1)}^k + \beta \lambda (1-\mu) G_{(1, i+1)}^k, &  \enspace i > k.
\end{cases}
\end{equation}

In the future, to simplify the notation, we will find it convenient to define 
\begin{equation}
\label{eq:di}
\sigma_i \triangleq 1 - \beta q - \beta^{i+1} p,
\end{equation}
and 
\[
\Gamma_{i}^k \triangleq 
\begin{cases}
\beta^{k-i} \Gamma_{k}^k, & \quad i < k \\
\displaystyle \frac{\beta \lambda}{\sigma_k}, & \quad i \geqslant k,
\end{cases}
\]

The next result gives closed-form expressions for evaluating {the work metric} $G_{(b, i)}^k$.
\begin{lemma}
\label{lma:Geval} We have
\[
G_{(0, i)}^k = 
\Gamma_{i}^k
\quad \textup{and} \quad
G_{(1, i)}^k = 
\begin{cases}
\Gamma_{i}^k, & \enspace i \leqslant k \\
1 + \mu \Gamma_{0}^k + (1-\mu) \Gamma_{i}^k, &  \enspace i > k.
\end{cases}
\]
\end{lemma}
\begin{proof}
See Appendix~\ref{s:wmeval}.
\end{proof}

The following result evaluates  the marginal work metric $g_{(1, i)}^k$. Note that we use the standard notation $x^+ \triangleq \max(x, 0)$.
\begin{lemma}
\label{lma:mgeval} 
\begin{align*}
g_{(1, i)}^k  & = 
1 - \mu (\beta^{(k-i)^+}  - \beta^{k}) G_{(0, k)}^k = 
\frac{1 - \beta (1 - (1-\beta^{(k-i)^+}) p)}{1 - \beta (1 - (1-\beta^k) p)} = \frac{1 - \beta q - \beta^{(k-i)^++1} p}{1 - \beta q - \beta^{k+1} p} \\
& = \frac{\sigma_{(k-i)^+}}{\sigma_k}, \quad i \geqslant 1.
\end{align*}
\end{lemma}
\begin{proof}
We can use the above results to represent $g_{(1, i)}^k$ as
\begin{align*}
g_{(1, i)}^k & = G_{(1, i)}^{\langle 1, k\rangle} - G_{(1, i)}^{\langle 0, k\rangle}  = 1 + \mu G_{(0, 0)}^k + (1-\mu) G_{(0, i)}^k - G_{(0, i)}^k = 
1 - \mu \big(\Gamma_{i}^k - \Gamma_{0}^k\big) \\
& =
\begin{cases}
 1 - \mu (\beta^{k-i}  - \beta^{k}) \Gamma_{k}^k, & \quad i \leqslant k \\
 1 - \mu (1 - \beta^{k}) \Gamma_{k}^k, & \quad i > k
\end{cases} 
= 
\begin{cases}
\displaystyle \frac{1 - \beta q - \beta^{k-i+1} p}{1 - \beta q - \beta^{k+1} p} = \frac{\sigma_{k-i}}{\sigma_k}, & \quad i \leqslant k \\ \\
\displaystyle \frac{1 - \beta}{1 - \beta q - \beta^{k+1} p} = \frac{\sigma_{0}}{\sigma_k}, & \quad i > k.
\end{cases}
\end{align*}
\end{proof}

We thus obtain that the model satisfies condition $($PCLI1$)$.

\begin{lemma}
\label{lma:pcli1}
$g_{(1, i)}^k  > 1-\beta$, and hence condition $($PCLI1$)$ holds.
\end{lemma}
\begin{proof}
It is immediately apparent from  Lemma~\ref{lma:mgeval} that $g_{(1, i)}^k$ is decreasing 
in $i$ for $i \leqslant k$, and hence 
\[
g_{(1, i)}^k \geqslant g_{(1, k)}^k = \frac{1 - \beta }{1 - \beta (1 - (1-\beta^k) p)} > 1 - \beta, \quad i \leqslant k.
\]

As for $i > k$, Lemma~\ref{lma:mgeval} shows that $g_{(1, i)}^k = g_{(1, k)}^k$, which completes the proof.
\end{proof}

\subsection{Cost Metric Analysis}
\label{s:cman}
We continue by analyzing cost and marginal cost metrics under the $k$-policy. As above,  we will find it convenient to include the fictitious AoI value $i = 0$.

We start {by} analyzing the cost metric, which satisfies the following linear equations:
\begin{equation}
\label{eq:F0ik}
F_{(0, i)}^k = c_i + \beta (1-\lambda) F_{(0, i+1)}^k + \beta \lambda F_{(1, i+1)}^k, \quad i  \geqslant 0
\end{equation}
and
\begin{equation}
\label{eq:F1ik}
F_{(1, i)}^k = 
\begin{cases}
c_i + \beta (1-\lambda) F_{(0, i+1)}^k + \beta \lambda F_{(1, i+1)}^k, & \enspace i \leqslant k \\
c_i + \beta (1-\lambda) \mu F_{(0, 1)}^k + \beta p F_{(1, 1)}^k  & \\
\quad + \beta (1-\lambda) (1-\mu) F_{(0, i+1)}^k + \beta \lambda (1-\mu) F_{(1, i+1)}^k, &  \enspace i > k.
\end{cases}
\end{equation}

In the future, to  simplify notation, we will find it convenient to define, for $i \geqslant 1$,
\begin{equation}
\label{eq:Didef}
C_i(q) \triangleq 
\begin{cases} \displaystyle 
\sum_{j=1}^{\infty} (\beta q)^{j-1} c_{i-1+j}, & \quad \textup{if} \quad q > 0 \\
\displaystyle
c_{i}, & \quad \textup{if} \quad q = 0.
\end{cases}
\end{equation}

{We further} define quantities $\Phi_{i}^k$ as follows. For $i = k$, 
\[
\Phi_{k}^k \triangleq 
 \frac{1}{\sigma_k} \Bigg[p \sum_{j=0}^{k-1} \beta^{j+1} c_j  +  (1-\beta q) \big(c_{k}   +  \beta C_{k+1}(q)\big)\Bigg];
\]
for $i < k$,
\[
\Phi_{i}^k \triangleq \sum_{j=i}^{k-1} \beta^{j-i} c_j + \beta^{k-i} \Phi_{k}^k;
\]
and, 
for $i > k$,
\begin{align*}
\Phi_{i}^k \triangleq 
 \begin{cases} \displaystyle (\beta q)^{k-i}
\bigg(\Phi_{k}^k - \beta p \frac{1-(\beta q)^{i-k}}{1 - \beta q} \Phi_{0}^k   - c_{k}  & \\
\displaystyle \qquad \qquad - \beta \sum_{j=1}^{i-k-1}  (\beta q)^{j-1} c_{k+j} - \beta p (\beta q)^{i-k-1} c_{i}\bigg), & \textup{ if } q > 0 \\
\displaystyle c_i + \beta c_{i+1} + \beta \Phi_{0}^k, & \textup{ if } q = 0.
\end{cases}
\end{align*}

The following result gives closed-form expressions for evaluating the {cost metric} $F_{(b, i)}^k$.
\begin{lemma}
\label{lma:Feval} We have
\[
F_{(0, i)}^k = 
\Phi_{i}^k
\quad \textup{and} \quad
F_{(1, i)}^k = 
\begin{cases}
F_{(0, i)}^k, & \quad i \leqslant k \\
\mu c_i + \mu F_{(0, 0)}^k + (1-\mu) F_{(0, i)}^k, &  \quad i > k.
\end{cases}
\]
\end{lemma}
\begin{proof}
\hl{See Appendix}~\ref{s:cmeval}.
\end{proof}

The next result evaluates the marginal cost metric $f_{(1, i)}^k$ in closed form. 

\begin{lemma}
\label{lma:mfeval}  We have
\begin{enumerate}[label=(\alph*)]
\item $f_{(1, i)}^k  = 
\mu \big(\Phi_{i}^k - \Phi_{0}^k - c_i\big);$
\item In particular, in the case $i = k$,
\[
f_{(1, i)}^i   = 
\frac{\mu}{\sigma_i}\bigg[\beta (1 - \beta^{i})(1-\beta q)  C_{i+1}(q) - (1 - \beta) \sum_{j=0}^{i} \beta^{j} c_j\bigg].
\]
\end{enumerate}
\end{lemma}
\begin{proof}
\hl{See Appendix}~\ref{s:cmeval}. 
\end{proof}

\subsection{{MP} Index Analysis and Condition \highlighting{$($PCLI2$)$}} 
\label{s:mpmapcli2}
{We next give a closed formula for the discounted MP index $m_{(1, i)}$ defined in~(\ref{eq:mpi}).}
 
 \begin{lemma}
\label{lma:dmpi}
We have
\begin{equation}
\label{eq:dmpi}
m_{(1, i)} = 
\begin{cases} \displaystyle
\mu \bigg[\frac{\beta (1 - \beta^{i})(1-\beta q)}{1 - \beta}   C_{i+1}(q) - \sum_{j=1}^{i} \beta^{j} c_j\bigg], \quad \textup{if} \quad q > 0 \\ \\ \displaystyle
\frac{\beta (1 - \beta^{i})}{1 - \beta} c_{i+1} - \sum_{j=1}^{i} \beta^{j} c_j, \quad \textup{if} \quad q = 0.
\end{cases}
\end{equation}
\end{lemma}
\begin{proof}
The result follows straightforwardly from Lemmas~\ref{lma:mgeval} and~\ref{lma:mfeval}, using~(\ref{eq:Didef}).
\end{proof}

{We are now ready to establish the PCL-indexability condition $($PCLI2$)$.}
\begin{lemma}
\label{lma:dmpicpcli2}
Under Assumption $\ref{ass:ci}$, $m_{(1, i)}$ is monotone nondecreasing, so condition $($PCLI2$)$ holds.
\end{lemma}
\begin{proof}
We have, 
using that $C_{i}(q) = c_{i} + \beta q C_{i+1}(q)$ and $c_i$ is nondecreasing, 
\begin{align*}
(1 - \beta^{i}) C_{i+1}(q) - 
(1 - \beta^{i-1}) C_{i}(q) & = 
 (1 - \beta^{i}) C_{i+1}(q) - 
 (1-\beta^{i-1}) \big(c_{i} + \beta q C_{i+1}(q)\big) \\
 & = \big(1-\beta^i  - (1-\beta^{i-1}) \beta q\big) C_{i+1}(q) - (1-\beta^{i-1}) c_{i} \\
 & = \big(1 - \beta q - \beta^{i} p\big) C_{i+1}(q) - (1-\beta^{i-1}) c_{i} \\
 & \geqslant \big(1 - \beta q - \beta^{i} p\big)  \sum_{j=1}^{\infty} (\beta q)^{j-1} c_{i} - (1-\beta^{i-1}) c_{i} \\
 & = \frac{(1-\beta ) \beta^{i-1}}{1-\beta  q} c_i.
\end{align*}

On the other hand,
\[
\sum_{j=1}^{i} \beta^{j} c_j - \sum_{j=1}^{i-1} \beta^{j} c_j = \beta^{i} c_i.
\]

Hence, we have, for $i \geqslant 2$,
\[
m_{(1, i)} - m_{(1, i-1)} \geqslant 
\frac{\mu}{1 - \beta}\bigg[\beta (1-\beta q)  \frac{(1-\beta ) \beta^{i-1}}{1-\beta  q} c_i - (1 - \beta) \beta^{i} c_i\bigg] = 0,
\]
which completes the proof.
\end{proof}

We thus obtain the following result.
\begin{proposition}
\label{pro:dwi} Under Assumption $\ref{ass:ci}$,
the discounted model is threshold-indexable, with Whittle index given by the MP index $m_{(1, i)}.$
\end{proposition}
\begin{proof}
The result follows from Lemmas~\ref{lma:pcli1} and~\ref{lma:dmpicpcli2} by applying Theorem~\ref{the:verif}.
\end{proof}

\section{PCL-Indexability Analysis: Average Cost Criterion}
\label{s:pcliaac}
This section {applies} Theorem~\ref{the:verif} to the present model under the average cost criterion.  {This} entails verifying conditions $($PCLI1$)$ and $($PCLI2$)$ for  average metrics that are counterparts of the discounted metrics above. 
Note that the PCL-indexability theory under the average criterion, based on Laurent series expansions, is discussed in  
\cite[Section \hl{5}]{nmaap01},
 \cite[Section~\hl{6.5}]{nmmp02}, and 
\cite[Section \hl{5.2} ]{nmmor06}, upon which we draw for the following analyses.

\subsection{Work Metric Analysis and Condition \highlighting{$($PCLI1$)$}} 
\label{s:awmacpcli1}
We start with condition $($PCLI1$)$, namely the positivity of the marginal work metric under 
threshold policies. 
We start by evaluating the average work metric under the $k$-policy, 
\begin{equation}
\label{eq:Gkac}
G_{(b, i)}^k \triangleq \lim_{T \to \infty} \, \frac{1}{T} \Ex_{(b, i)}^k\Bigg[\sum_{t=0}^T A(t)\Bigg].
\end{equation}

The following result {evaluates $G_{(b, i)}^k$ in closed form, which shows that it does not 
depend on the initial state $(b, i)$; hence, we write it as $G^k$}. This follows from standard 
results {relating} discounted and average cost metrics in MDPs. See, e.g.,  \cite[Theorem 
\hl{8.10.7}]{put94}.  
{Note that, below, we incorporate the discount factor $\beta$ in the notation for the 
discounted metrics considered above to distinguish them from their average criterion 
counterparts. Thus, e.g., we write the discounted work metric as $G_{\beta, (b, i)}^k$.}
 
\begin{lemma}
\label{lma:aGeval}  
We have
\[
G^k = G_{(b, i)}^k = \lim_{\beta \to 1} \, (1 - \beta) G_{\beta, (b, i)}^k = \frac{\lambda }{k p + 1}.
\]
\end{lemma}
\begin{proof}
The result follows by using the expressions for $G_{\beta, (b, i)}^k$ in Lemma 
\ref{lma:Geval}.
\end{proof}

The next result {evaluates}  the average marginal work metric $g_{(1, i)}^k$ based on the 
above formulae for its discounted counterpart, which we now write as $g_{\beta, (1, i)}^k$.

\begin{lemma}
\label{lma:dmgeval} We have
\[
g_{(1, i)}^k  = \lim_{\beta \to 1} \, g_{\beta, (1, i)}^k = \frac{(k-i)^+ p+1}{k p+1}, \quad i \geqslant 1.
\]
\end{lemma}
\begin{proof}
The result follows by using the expression for $g_{\beta, (1, i)}^k$ in Lemma
\ref{lma:mgeval}.
\end{proof}

We thus obtain satisfaction of condition $($PCLI1$)$ under the average criterion.

\begin{lemma}
\label{lma:acpcli1}
$g_{(1, i)}^k  > 0$, and hence condition $($PCLI1$)$ holds.
\end{lemma}
\begin{proof}
The result follows directly from Lemma~\ref{lma:dmgeval}.
\end{proof}

\subsection{Cost Metric Analysis}
\label{s:dcman}
We continue by analyzing the average cost and marginal cost metrics under the $k$-policy 
for $k \geqslant 0$, starting with the average
cost metric, given by
\begin{equation}
\label{eq:Fkac}
F_{(b, i)}^k \triangleq \lim_{T \to \infty} \, \frac{1}{T} \Ex_{(b, i)}^k\Bigg[\sum_{t=0}^T c_{X(t)}\Bigg].
\end{equation}

In the future, we will find it convenient to define (cf.\ (\ref{eq:Didef})), for $i \geqslant 1$,
\begin{equation}
\label{eq:dDidef}
C_i(q) \triangleq 
\begin{cases} \displaystyle 
\sum_{j=1}^{\infty} q^{j-1} c_{i-1+j}, & \quad \textup{if} \quad q > 0 \\
\displaystyle
c_{i}, & \quad \textup{if} \quad q = 0,
\end{cases}
\end{equation}

Let us now define quantities $\Phi^k$ as follows: for any $i \geqslant 0$,
\[
\Phi^k \triangleq \lim_{\beta \to 1} \, (1-\beta) \Phi_{\beta, i}^k = 
 \frac{1}{k p + 1} \Bigg[p \sum_{j=0}^{k-1} c_j  +  p \big(c_{k}   +  \beta C_{k+1}(q)\big)\Bigg].
\]

The following result {evaluates $F_{(b, i)}^k$ in closed form, which shows that it does not 
depend on the initial state $(b, i)$; hence, we write it as $F^k$}. This follows from standard 
results (see, e.g., \cite[Theorem \hl{8.10.7} ]{put94}). Below, we denote the discounted cost 
metric as $F_{\beta, (b, i)}^k$. 

\begin{lemma}
\label{lma:aFeval}  We have
\[
F^k = F_{(b, i)}^k = \lim_{\beta \to 1} \, (1-\beta) F_{\beta, (b, i)}^k = 
\Phi^k.
\]
\end{lemma}
\begin{proof}
The result follows by using the formulae for $F_{\beta, (b, i)}^k$ in Lemma 
\ref{lma:Feval}.
\end{proof}

The next result gives a closed-form expression for  the average marginal cost metric $f_{(1, 
i)}^i$, which is the only one needed for our analyses. Below, we denote  the discounted 
marginal cost metric as $f_{\beta, (1, i)}^k$. 

\begin{lemma}
\label{lma:amfeval} We have
\[
f_{(1, i)}^i   = \lim_{\beta \to 1} \, f_{\beta, (1, i)}^i = 
\frac{\mu}{i p+1}\bigg(i  p  \, C_{i+1}(q) -  \sum_{j=0}^{i} c_j\bigg).
\]
\end{lemma}
\begin{proof}
The result follows by using the expression for $f_{\beta, (1, i)}^i$ in Lemma~\ref{lma:mfeval}(b).
\end{proof}

\subsection{MP Index Analysis and Condition \highlighting{$($PCLI2$)$}} 
\label{s:ampmapcli2}
We are now  ready to address the PCL-indexability condition $($PCLI2$)$, namely, the 
nondecreasingness of the  MP \emph{index} 
defined by~(\ref{eq:mpi}), using the marginal metrics for the average cost criterion.
The next result gives closed formulae for such an average criterion MP index by drawing 
on the discounted MP index derived above, which we now denote as $m_{\beta, (1, i)}$.
 \begin{lemma}
\label{lma:ampicpcli2}
\begin{equation}
\label{eq:ampi}
m_{(1, i)} = 
\mu \bigg[i p   \, C_{i+1}(q) - \sum_{j=1}^{i} c_j\bigg]
= 
\begin{cases} \displaystyle
\mu \bigg[i p    \sum_{j=1}^{\infty} q^{j-1} c_{i+j} - \sum_{j=1}^{i} c_j\bigg], & \quad \textup{if} \quad q > 0 \\ \\ \displaystyle
i p   \, c_{i+1} - \sum_{j=1}^{i} c_j\, & \quad \textup{if} \quad q = 0.
\end{cases}
\end{equation}
\end{lemma}
\begin{proof}
The result follows by taking the limit 
$m_{(1, i)} = \lim_{\beta \to 1} \, m_{\beta, (1, i)}$ using the formulae in Lemma~\ref{lma:dmpi} for $m_{\beta, (1, i)}$.
\end{proof}

We next establish satisfaction of condition $($PCLI2$)$.
\begin{lemma}
\label{lma:apcli2}
Under Assumption $\ref{ass:ci}$, the average criterion MP index $m_{(1, i)}$ is monotone nondecreasing, so $($PCLI2$)$ holds.
\end{lemma}
\begin{proof}
The result follows from the nondecreasingness of the discounted MP index in Lemma~\ref{lma:dmpicpcli2}, since $m_{(1, i)} = \lim_{\beta \to 1} \, m_{\beta, (1, i)}$.
\end{proof}

We can now establish the model's indexability and evaluate its Whittle index.

\begin{proposition}
\label{pro:awi}
Under Assumption $\ref{ass:ci}$,
the model is threshold-indexable under the average criterion, with Whittle index given by the MP index $m_{(1, i)}.$
\end{proposition}
\begin{proof}
The result follows from Lemmas~\ref{lma:acpcli1} and~\ref{lma:ampicpcli2} using 
Theorem~\ref{the:verif} (in its average criterion version).
\end{proof}

\section{Index Dependence on Model Parameters and Evaluation for Special Cost Functions}
 \label{s:dwimp}
 This section exploits the Whittle index formulae obtained above for two interrelated 
 purposes: exploring the dependence of the index on model parameters, drawing 
 corresponding insights, and specializing the index formula to relevant special cost 
 functions.  

\subsection{Monotonicity of the Whittle Index on the Arrival Rate}
\label{s:dmiar}
We have the following result.

\begin{proposition}
\label{pro:dmiar}
Both
the discounted Whittle index $m_{\beta, (1, i)}$ and the average Whittle index $m_{(1, i)}$ are nonincreasing in the packet arrival probability $\lambda$.
\end{proposition}
\begin{proof}
For the discounted criterion,
we use the expression of the index in~(\ref{eq:dmpi}), along with 
\begin{align*}
\frac{d}{d \lambda} C_{\beta, i+1}(q) & = \frac{d}{d \lambda} \sum_{j=1}^{\infty} (\beta q(\lambda))^{j-1} c_{i+j} = 
\frac{d}{d \lambda} \sum_{j=2}^{\infty} (\beta q(\lambda))^{j-1} c_{i+j} \\
& = \beta \bigg(\sum_{j=2}^{\infty} (j-1) (\beta q(\lambda))^{j-2}   c_{i+j}\bigg) \frac{d}{d \lambda} q(\lambda) \\
& = - \beta \mu \sum_{j=1}^{\infty} j (\beta q)^{j-1}   c_{i+j+1},
\end{align*}
to obtain 
\begin{align*}
\frac{d}{d \lambda} m_{\beta, (1, i)} = -
\frac{(\beta \mu)^2 (1 - \beta^{i})(1-\beta q)}{1 - \beta}     \sum_{j=1}^{\infty} j (\beta q)^{j-1}   c_{i+j+1} \leqslant 0.
\end{align*}

For the average criterion, we use the expression of the index in~(\ref{eq:ampi}), along with 
\[
\frac{d}{d \lambda} C_{i+1}(q) = -\mu \sum_{j=1}^{\infty} j q^{j-1}   c_{i+j+1},
\]
to obtain 
\[
\frac{d}{d \lambda} m_{(1, i)} = 
-i p \mu ^2    \sum_{j=1}^{\infty} j q^{j-1}   c_{i+j+1} 
\leqslant 0.
\]
\end{proof}

Proposition~\ref{pro:dmiar} provides the following insight regarding the packet 
transmission priority under Whittle's index policy: other things being equal, the 
policy gives higher (or, more precisely, not lower) transmission priority to users whose 
packets arrive at lower rates. 
This is consistent with the numerical results in 
\cite[Section \hl{7.1} ]{hsuetal20} on the 
properties of optimal scheduling policies for the special case with reliable channels under 
the average criterion.

It is readily verified with examples that, for general cost functions, the Whittle index does  
not need to be monotonic on the  transmission success probability $\mu$.
Yet, some monotonicity properties can be identified in the case of special cost functions, as discussed below.

\subsection{Whittle Index Formulae for Special Cost Functions}
\label{s:dscf}
We next consider evaluation of the Whittle index for special cost functions, which yields formulae that can be implemented with greater efficiency than the general formulae above.
We shall consider the cases of linear, quadratic, and threshold-type AoI costs.

\subsubsection{Linear AoI Cost}
\label{s:dlaoic}
We start with the case of a linear cost $c_i \triangleq c  i$, with $c > 0$. The discounted index reduces to
\begin{equation}
\label{eq:dwilin}
m_{\beta, (1, i)} = 
\frac{\beta c \mu}{1 - \beta}\bigg[i-\frac{\beta 
   \left(1-\beta ^i\right)
   p}{(1-\beta ) (1-\beta 
   q)}\bigg],
\end{equation}
and the average criterion index simplifies to
\begin{equation}
\label{eq:awilin}
m_{(1, i)} = \frac{c \mu}{2} \left(i^2+\frac{1+q}{p} i\right) = 
   c  \mu i  \bigg(\frac{i - 1}{2} + \frac{1}{p}\bigg).
\end{equation}

Even though, for general costs, the Whittle index is not necessarily monotonic in the 
probability of successful transmission $\mu$, in this special case it turns out to be 
increasing in $\mu$ (for $i \geqslant 2$ under the average criterion, since $m_{(1, 1)} = c 
/\lambda$).

\begin{proposition}
\label{pro:wiiinmu} For linear costs, 
the discounted Whittle index $m_{\beta, (1, i)}$ is increasing in $\mu$. The same holds for the average Whittle index $m_{(1, i)}$ for $i \geqslant 2$.
\end{proposition}
\begin{proof}
Under the discounted cost criterion, we have
\[
\frac{d}{d \mu} m_{\beta, (1, i)} = 
\frac{\beta c}{1 - \beta}\bigg[i-\frac{\beta  \left(1-\beta ^i\right) (2-\beta 
   (2-p)) p}{(1-\beta ) (1-\beta  q)^2}\bigg] > 0, 
\]
which follows because
\begin{equation}
\label{eq:chainineqmu}
\frac{\beta  (2-\beta 
   (2-p)) p}{(1-\beta ) (1-\beta  q)^2} \leqslant \frac{1}{1-\beta} < \frac{i}{1-\beta ^i}.
\end{equation}

Note that~(\ref{eq:chainineqmu}) holds because $i/(1-\beta ^i)$ increases with $i$, and the left-most expression in~(\ref{eq:chainineqmu}) attains the value $1/(1-\beta)$ for $p = 1$ and increases with $p$ (recall that $q \triangleq 1 - p$), since
\[
\frac{d}{d p} \frac{\beta  (2-\beta 
   (2-p)) p}{(1-\beta ) (1-\beta  q)^2}  = \frac{2 \beta (1-\beta)}{(1 - \beta q)^3} > 0.
\]

Under the average cost criterion, we have, for $i \geqslant 2$, 
\[
\frac{d}{d \mu} m_{(1, i)} = \frac{c i (i-1)}{2} > 0.
\]

This proves the result.
\end{proof}

Hence, for linear AoI costs, other things being equal, Whittle's policy gives higher 
transmission priority to users for which packet transmissions are more likely to succeed.

We next graphically illustrate the above results for the base instance with linear cost $c_i 
\triangleq i$ and the parameters  $\lambda = 0.7$, $\mu = 0.8$, and $\beta = 0.8$ for the 
states $(1, 1)$, $(1, 2)$, and $(1, 3)$.
Figure~\ref{fig:mpiaoiL} plots the corresponding Whittle index vs.\ the AoI under the 
discounted and average cost criteria. The plots show that the discounted index grows nearly linearly in the AoI, in contrast to the quadratic growth of the average criterion index. 

Regarding the dependence on $\lambda$, Figure~\ref{fig:mpilambdaL} plots the 
discounted and the average criterion Whittle indices vs.\ $\lambda$ for AoI values $i = 1, 
2, 3$. The plot is consistent with Proposition~\ref{pro:dmiar}, as the Whittle index is 
decreasing in $\lambda$.
A significant difference is that, as $\lambda \to 0$, 
the discounted index converges to a finite limit, easily seen to be 
$\lim_{\lambda \to 0} m_{\beta, (1, i)}  = \beta c \mu i /(1-\beta)$,
whereas the average 
criterion index diverges to infinity. Thus, under the average cost criterion, the increase in 
priority becomes steeper as the packet arrival probability becomes lower. 

Consider now the dependence on $\mu$. Figure~\ref{fig:mpimuL} plots the discounted 
and the average criterion Whittle indices vs.\ $\mu$ for AoI values $i = 1, 2, 3$. The plot is 
consistent with Proposition~\ref{pro:wiiinmu}, as the index increases with $\mu$ 
(except for $i = 1$ under the average criterion, where it remains constant as $\mu$ varies).
Note that the increase is linear under the average criterion.

\begin{figure}[H]
\centering
\includegraphics[height=2.3in]{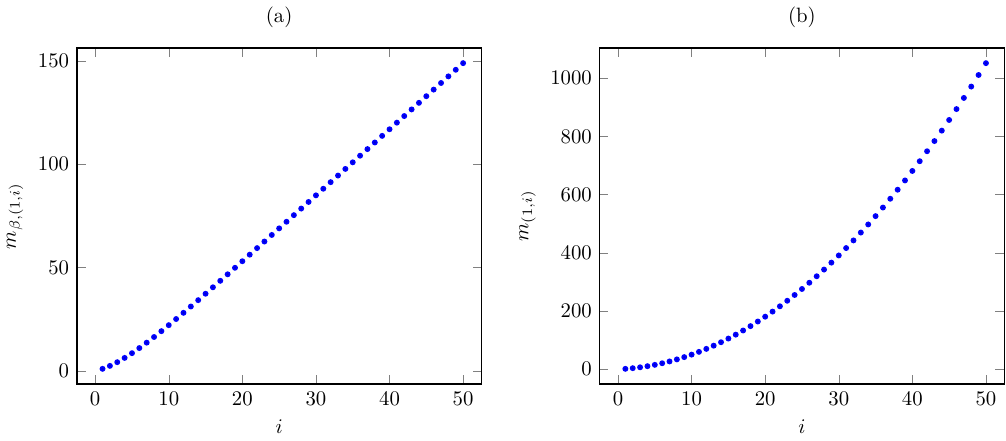}
\caption{
Discounted (\textbf{a}) and average (\textbf{b}) criterion Whittle index vs.\ AoI $i$ for linear AoI 
 cost.}
\label{fig:mpiaoiL}
\end{figure}

\vspace{-12pt}

\begin{figure}[H]
\centering
\includegraphics[height=2.3in]{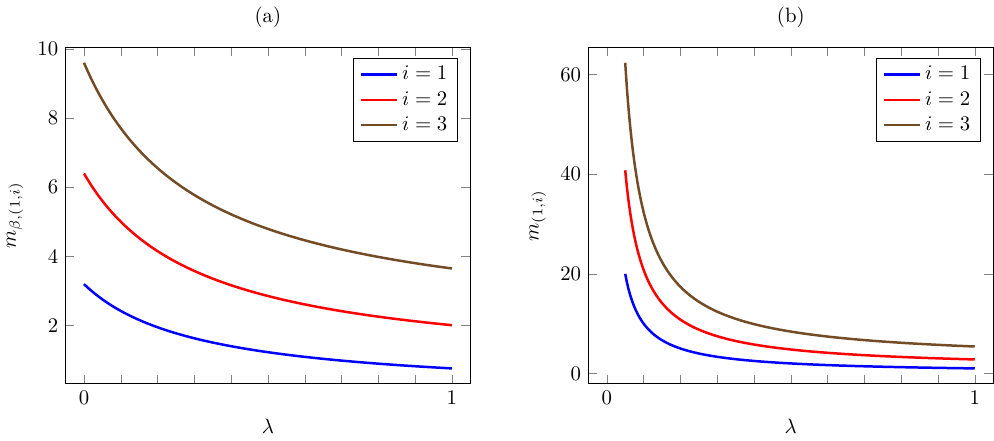}
\caption{\hl{Discounted} 
 (\textbf{a}) and average (\textbf{b}) criterion Whittle index vs.\ $\lambda$ for linear AoI 
 cost.}
\label{fig:mpilambdaL}
\end{figure}

\vspace{-12pt}
\begin{figure}[H]
\centering
\includegraphics[height=2.3in]{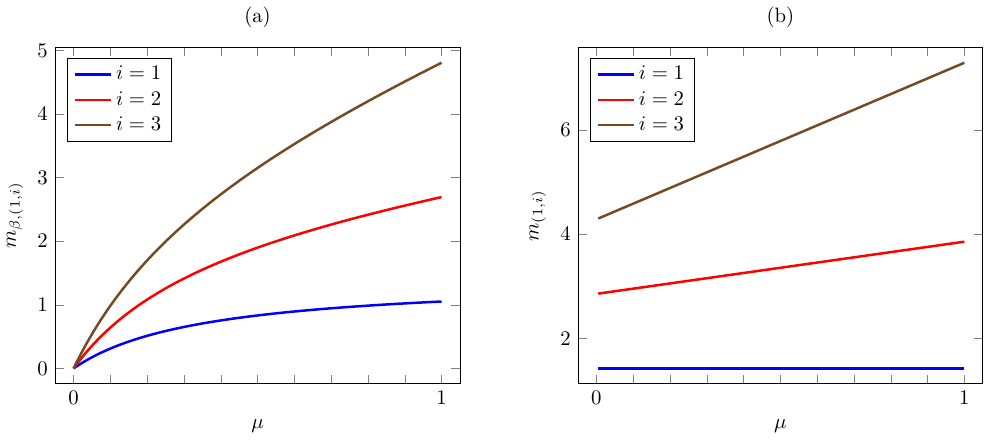}
\caption{\hl{Discounted} 
 (\textbf{a}) and average (\textbf{b}) criterion Whittle indices vs.\ $\mu$ for linear AoI 
 cost.}
\label{fig:mpimuL}
\end{figure}

\subsubsection{{Quadratic} AoI Cost}
\label{s:dqaoic}
We next turn to the case of a 
quadratic cost function $c_i \triangleq c i^2$. Then, the discounted Whittle index has 
the evaluation 
\begin{equation}
\label{eq:dwiq}
m_{\beta, (1, i)} = 
\frac{\beta c \mu}{1 - \beta}\bigg[i^2+\frac{2 \left(1-\beta +p \beta ^{i+1}\right)}{(1-\beta )
   (1-\beta  q)} i - \frac{\beta p  \left(1-\beta ^i\right) (3-\beta  (1+(1+\beta ) q))}{(1-\beta )^2
   (1-\beta  q)^2}\bigg],
\end{equation}
whereas the average criterion Whittle index reduces to
\begin{equation}
\label{eq:awiq}
m_{(1, i)} = 
c \mu \bigg[\frac{2}{3} i^3 +\frac{4-(1+q)^2}{2
   p^2}  i^2 +\frac{21-(3+p)^2}{6
   p^2}  i\bigg].
\end{equation}

Considering the same instance as above, Figure~\ref{fig:mpiaoiQ} plots the corresponding Whittle index vs.\ the AoI under the 
discounted and average cost criteria. The plots show that the Whittle index growth with the AoI is much steeper under the average criterion (having cubic growth) than under the discounted criterion (having  quadratic growth). 

\begin{figure}[H]
\centering
\includegraphics[height=2.3in]{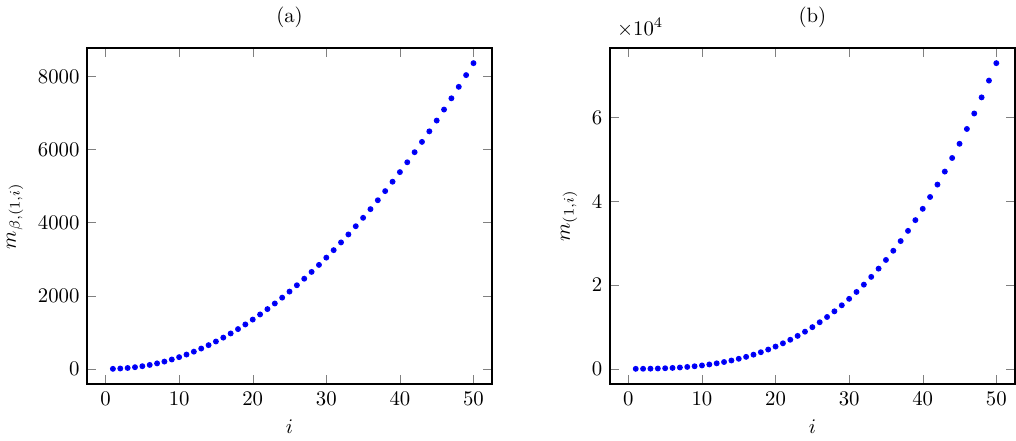}
\caption{
Discounted (\textbf{a}) and average (\textbf{b}) criterion Whittle index vs.\ AoI $i$ for quadratic AoI 
 cost.}
\label{fig:mpiaoiQ}
\end{figure}

As for the dependence on $\lambda$, the plots in Figure~\ref{fig:mpilambdaQ} are also consistent with Proposition~\ref{pro:dmiar}.
As was the case for the linear cost case, we see that, as $\lambda \to 0$, 
the discounted index converges to a finite limit, whereas the average criterion index 
diverges to infinity. 

\begin{figure}[H]
\centering
\includegraphics[height=2.3in]{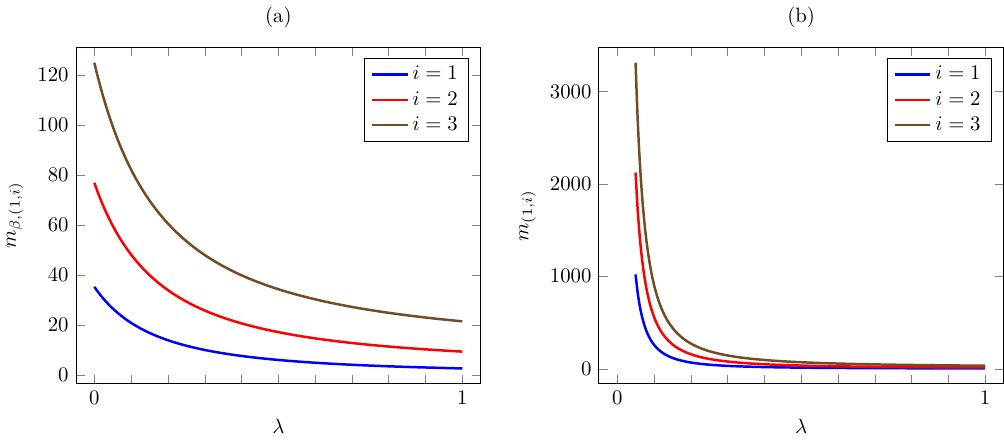}
\caption{\hl{Discounted} 
 (\textbf{a}) and average (\textbf{b}) criterion Whittle indices vs.\ $\lambda$ for 
 quadratic AoI cost.}
\label{fig:mpilambdaQ}
\end{figure}

Consider now the dependence on $\mu$. Figure~\ref{fig:mpimuQ} plots the discounted 
and the average criterion Whittle index vs.\ $\mu$ for the same instance.
The plot reveals a strikingly different behavior of the discounted and average criterion 
Whittle indices.
Thus, for $i = 1$, the discounted index appears to first increase and then to decrease in $\mu$. 
To verify it, note that
\[
m_{\beta, (1, 1)} = \frac{\beta  (3-\beta  q) c \mu }{(1-\beta 
   q)^2}, 
\]
and 
\[
\frac{d}{d \mu} m_{\beta, (1, 1)} = \frac{\beta  c (3-\beta  (4-\beta +(1+\beta )
   p))}{(1-\beta q)^3}.
\]
Thus, for the given instance, we have
\[
\frac{d}{d \mu} m_{\beta, (1, 1)} = \frac{100 (55-126 \mu )}{(5+14 \mu )^3},
\]
and hence $m_{\beta, (1, 1)}$ increases in $\mu$ for $0 < \mu < 55/126$ and decreases in $\mu$ for $55/126 < \mu < 1$.

\begin{figure}[H]
\centering
\includegraphics[height=2.3in]{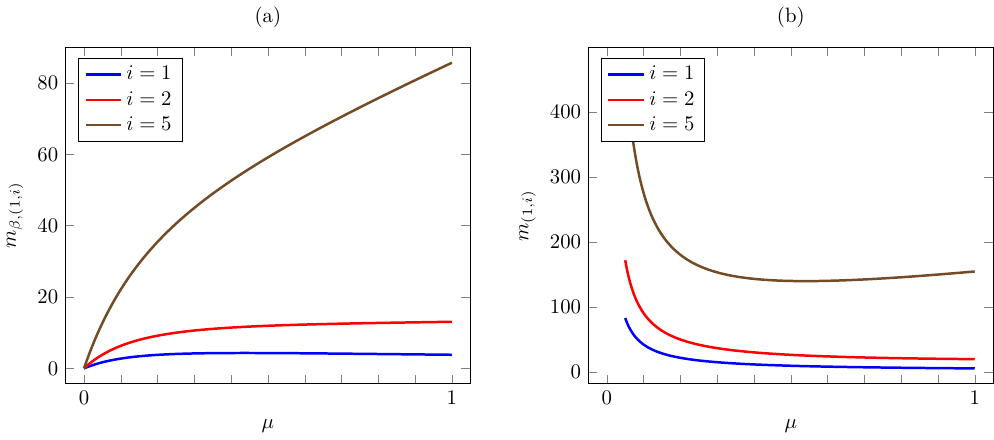}
\caption{\hl{Discounted} 
 (\textbf{a}) and average (\textbf{b}) criterion Whittle indices vs.\ $\mu$ for quadratic 
 AoI cost.}
\label{fig:mpimuQ}
\end{figure}

On the other hand, the plot shows that the discounted Whittle index increases with $\mu$ for $i = 2$ and
for $i = 5$.  It is readily verified analytically that such is the case in this instance for AoI 
values $i \geqslant 2$.

In contrast, the plot shows that the average criterion Whittle index $m_{(1, 1)}$ decreases with $\mu$. 
We can readily verify this, noting that 
\begin{equation}
\label{eq:ddmum1iq}
\frac{d}{d \mu} m_{(1, i)} = 
c \bigg[\frac{2
   i^3}{3}-\frac{i^2}{2}-\left(\frac{1}{6}+\frac{2}{p^
   2}\right) i\bigg] = c i \bigg[\frac{2
   }{3} i^2-\frac{1}{2} i-\left(\frac{1}{6}+\frac{2}{p^
   2}\right)\bigg],
\end{equation}
and hence 
\[
\frac{d}{d \mu} m_{(1, 1)} = -\frac{2c}{p^2} < 0.
\]

Furthermore, we have, for $0 < p < 1$,
\[
\frac{d}{d \mu} m_{(1, 2)} = c \bigg(3 - \frac{4}{p^2}\bigg) < 0,
\]
and hence $m_{(1, 2)}$ also decreases with $\mu$, consistently with the plot.

As for $m_{(1, 5)}$, the plot shows that it first decreases and then increases with $\mu$. 
Using~(\ref{eq:ddmum1iq}), it is readily verified analytically that such behavior holds 
provided $\lambda > 1/\sqrt{7}$, since
 \[
\frac{d}{d \mu} m_{(1, 5)} = c \bigg(70 - \frac{10}{\lambda^2 \mu^2}\bigg),
\]
and hence $m_{(1, 5)}$ has a unique minimum at $\mu^* = 1/(\lambda \sqrt{7}) \in (0, 1)$ for such $\lambda$.

The same qualitative behavior holds for $m_{(1, i)}$ where $i \geqslant 3$, first decreasing 
and then increasing with $\mu$, provided that
\[
\lambda > \frac{2 \sqrt{3}}{\sqrt{4 i^2-3 i-1}}
\]
since in such cases $m_{(1, i)}$ has a unique minimum as a function of $\mu$ at \vspace{-6pt}
\[
\mu^* = \frac{2 \sqrt{3}}{\lambda \sqrt{4 i^2-3 i-1}} \in (0, 1).
\]

Note also that as $\mu \to 0$,  $m_{\beta, (1, i)} \to 0$, whereas $m_{(1, i)} \to \infty$.

\subsubsection{Threshold-Type AoI Cost}
\label{s:dtaoic}
Finally, consider the case of a threshold-type cost function, with 
$c_i  \triangleq c \, 1_{\{i > k\}}$. Then, it is readily verified that the discounted Whittle index is given by
\begin{equation}
\label{eq:dwithr}
m_{\beta, (1, i)} = 
\begin{cases} \displaystyle
\frac{\beta c \mu}{1 - \beta}  (1 - \beta^{i})  (\beta q)^{k-i}, & \quad \textup{if} \quad i < k \\ \\ \displaystyle
\frac{\beta c \mu}{1 - \beta} (1 - \beta^{k}), & \quad \textup{if} \quad i \geqslant k,
\end{cases}
\end{equation}
whereas the average criterion index has the evaluation
\begin{equation}
\label{eq:awithr}
m_{(1, i)} = 
\begin{cases} \displaystyle
c \mu  i  q^{k-i}, & \quad \textup{if} \quad i < k \\ \\ \displaystyle
c \mu  k, & \quad \textup{if} \quad i \geqslant k.
\end{cases}
\end{equation}

Considering the same instance as above, with $k = 10$, Figure~\ref{fig:mpiaoiT} plots the Whittle index vs.\ the AoI under the 
discounted and average cost criteria. The plots show that the Whittle index growth with the AoI is slightly steeper for $i  < k$ under the average criterion than under the discounted criterion. 
For $i \geqslant k$, both indices remain constant as the AoI $i$ grows.

\begin{figure}[H]
\centering
\includegraphics[height=2.3in]{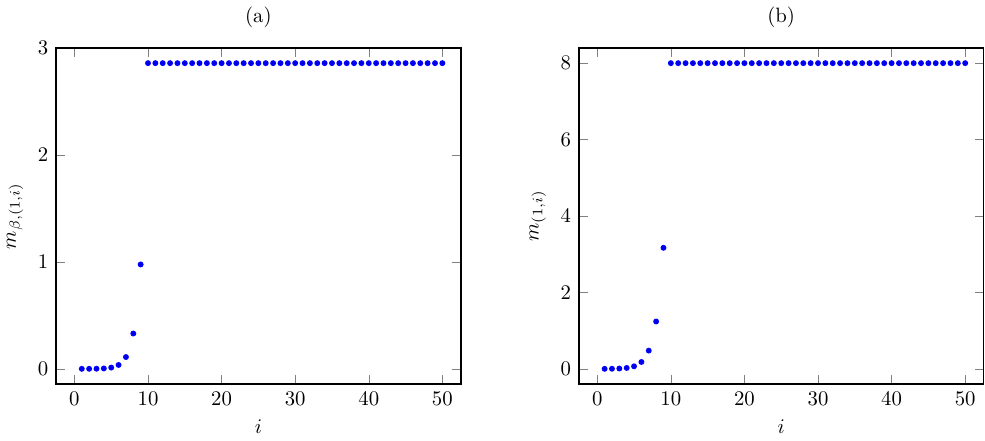}
\caption{\hl{Discounted} 
(\textbf{a}) and average (\textbf{b}) criterion Whittle index vs.\ AoI $i$ for threshold-type AoI cost.}
\label{fig:mpiaoiT}
\end{figure}

Regarding the dependence on $\lambda$, 
we have
\[
\frac{d}{d \lambda} m_{\beta, (1, i)} = 
\begin{cases} \displaystyle
- \frac{\beta^2 c \mu^2}{1 - \beta}  (1 - \beta^{i})   (k-i)  (\beta q)^{k-i-1}, & \quad \textup{if} \quad i < k \\ \\ \displaystyle
0, & \quad \textup{if} \quad i \geqslant k,
\end{cases}
\]
and
\[
\frac{d}{d \lambda} m_{(1, i)} = 
\begin{cases} \displaystyle
-c \mu^2 i (k-i) q^{k-i-1}, & \quad \textup{if} \quad i < k \\ \\ \displaystyle
0, & \quad \textup{if} \quad i \geqslant k.
\end{cases}
\]

Figure~\ref{fig:mpilambdaT} plots the discounted and the average criterion Whittle 
indices vs.\ $\lambda$. The plot shows that both indices are decreasing in $\lambda$ for 
$i < k$ and remain constant for $i \geqslant k$. Note that the indices decrease linearly 
with $\lambda$ for $i = k-1$.

\begin{figure}[H]
\centering
\includegraphics[height=2.3in]{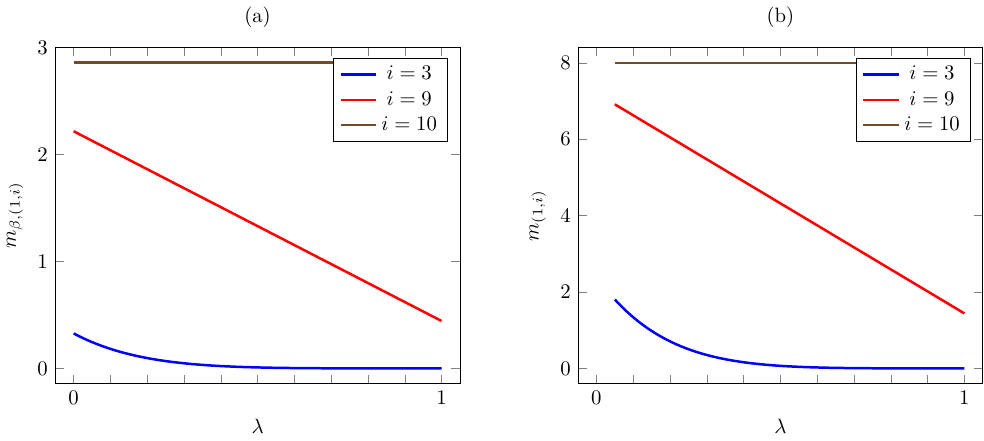}
\caption{\hl{Discounted} 
 (\textbf{a}) and average (\textbf{b}) criterion Whittle indices vs.\ $\lambda$ for 
 threshold-type AoI cost.}
\label{fig:mpilambdaT}
\end{figure}

As for the dependence of the Whittle index on $\mu$, we have, for the discounted index,
\[
\frac{d}{d \mu} m_{\beta, (1, i)} = 
\begin{cases} \displaystyle -
c \frac{\beta^2 (1 - \beta^{i})  (\beta q)^{k-i-1}}{1 - \beta}   ((k-i+1) p-1), & \quad \textup{if} \quad i < k \\ \\ \displaystyle
 c \frac{\beta (1 - \beta^{k})}{1 - \beta}, & \quad \textup{if} \quad i \geqslant k,
\end{cases}
\]
and, for the average criterion index, 
\[
\frac{d}{d \mu} m_{(1, i)} = 
\begin{cases} \displaystyle 
-c  i  q^{k-i-1} ((k-i+1) p-1), & \quad \textup{if} \quad i < k \\ \\ \displaystyle
 c  k, & \quad \textup{if} \quad i \geqslant k.
\end{cases}
\]

Hence, for $i < k$, both indices $m_{\beta, (1, i)}$ and $m_{(1, i)}$ decrease with $\mu$ whenever $(k-i+1) p >1$, i.e., for 
$\mu > 1 / (\lambda (k-i+1))$, and increase with $\mu$ for sufficiently low $\mu$ values, 
in particular for $\mu < 1 / (\lambda (k-i+1))$.
In contrast, for $i \geqslant k$, both indices  increase linearly in $\mu$.
Such behavior is illustrated in Figure~\ref{fig:mpimuT}.

\begin{figure}[H]
\centering
\includegraphics[height=2.3in]{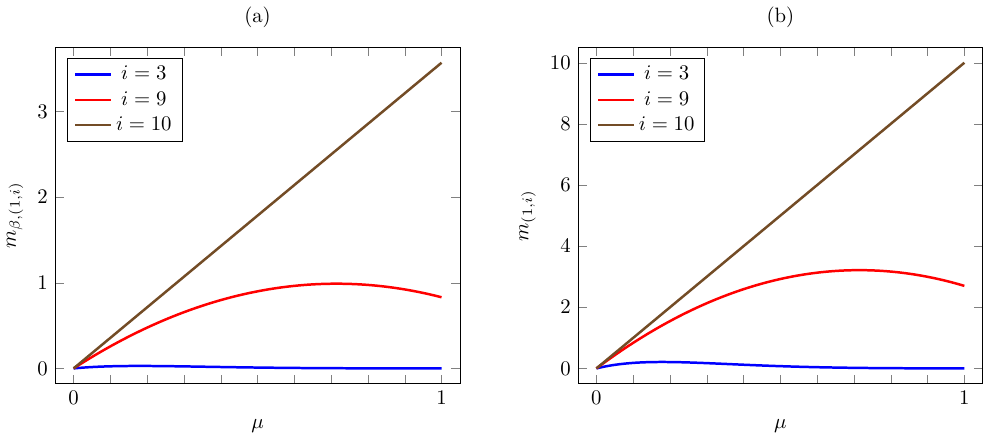}
\caption{\hl{Discounted} 
 (\textbf{a}) and average (\textbf{b}) criterion Whittle indices vs.\ $\mu$ for 
 threshold-type AoI cost.}
\label{fig:mpimuT}
\end{figure}

\section{Discussion}
\label{s:discuss}
This paper has addressed a broad model for scheduling AoI updates simultaneously 
incorporating the features of  random packet arrivals, unreliable channels, and general 
nondecreasing AoI costs in a no-buffer setting. 
Indexability (existence of the Whittle index) had only been established in prior work for 
special cases of this model and only under the average cost criterion using the prevailing 
approach for indexability analysis, based on first establishing optimality of threshold 
policies and then exploiting the monotonicity property of the optimal threshold.
We have proven that the broad model addressed is indexable and have evaluated its 
Whittle index in closed form, both under the discounted and the average criterion, by 
carrying out a PCL-indexability analysis, an approach that has been developed by the 
author over the last two decades.
By demonstrating its effectiveness in a relevant model that had not previously yielded to the prevailing analysis approach, 
we aim to bring such an approach to the attention of researchers in the area so that they 
can use it to analyze other models. 
We have further drawn qualitative insights on how scheduling priority under Whittle's 
index policy depends on models parameters by exploiting the closed formulae for 
Whittle's index for general costs and for special cost functions. 
It remains to numerically test the performance of Whittle's index policy in the present 
model.
Future research directions include carrying out a PCL-indexability analysis in more complex AoI scheduling models, such as those incorporating buffers.

\section*{Funding}
This work was supported in part by grant PID2019-109196GB-I00 funded by MCIN/AEI/10.13039/501100011033.

\section*{Abbreviations}
The following abbreviations are used in this manuscript:\\

\noindent 
\begin{tabular}{@{}ll}
AoI & Age of Information \\
{BS} & {Base station} \\
{CS} & {Complementary slackness} \\
MABP & Multi-armed bandit problem\\
MDP & Markov decision process \\
{MP} & {Marginal productivity} \\
{PCL} & {Partial conservation laws} \\
{PCLI} & {PCl-indexability} \\
RMABP & Restless multi-armed bandit problem
\end{tabular}


\appendix
\section[\appendixname~\thesection]{Work Metric Evaluation}
\label{s:wmeval}
\begin{proof}[Proof of Lemma $\ref{lma:Geval}$]
The fact that $G_{(1, i)}^k$ can be expressed in terms of the $G_{(0, j)}^k$ through 
\begin{equation}
\label{eq:G1ikG0ik}
G_{(1, i)}^k = 
\begin{cases}
G_{(0, i)}^k, & \enspace i \leqslant k \\
1 + \mu G_{(0, 0)}^k + (1-\mu) G_{(0, i)}^k, &  \enspace i > k.
\end{cases}
\end{equation}
follows straightforwardly from~(\ref{eq:G0ik}) and~(\ref{eq:G1ik}).

On the other hand, we have the following recursions for $G_{(0, i)}^k$:
\begin{equation}
\label{eq:G0ikrecgeqk}
\begin{split}
G_{(0, i)}^k & = \beta G_{(0, i+1)}^k, \quad i < k \\
G_{(0, i)}^k & = \beta (\lambda + p G_{(0, 0)}^k) + \beta q  G_{(0, i+1)}^k, \enspace i \geqslant k.
\end{split}
\end{equation}

The first recursion follows from~(\ref{eq:G0ik}) and~(\ref{eq:G1ikG0ik}) as, for $i < k$, 
\begin{align*}
G_{(0, i)}^k & = \beta (1-\lambda) G_{(0, i+1)}^k + \beta \lambda G_{(1, i+1)}^k = \beta (1-\lambda) G_{(0, i+1)}^k + \beta \lambda G_{(0, i+1)}^k = \beta G_{(0, i+1)}^k.
\end{align*}

The second recursion also follows from~(\ref{eq:G0ik}) and~(\ref{eq:G1ikG0ik}) because, for $i \geqslant k$ (recall that $p \triangleq \lambda \mu$, $q \triangleq 1 - p$):
\vspace{-12pt}
\begin{adjustwidth}{-1cm}{0cm}
\begin{align*}
G_{(0, i)}^k & = \beta (1-\lambda) G_{(0, i+1)}^k + \beta \lambda G_{(1, i+1)}^k  \\
& = \beta (1-\lambda) G_{(0, i+1)}^k + \beta \lambda \big(1 + \mu G_{(0, 0)}^k + (1-\mu) G_{(0, i+1)}^k\big)  \\
& =  \beta (\lambda + p G_{(0, 0)}^k) + \beta q  G_{(0, i+1)}^k.
\end{align*}
\end{adjustwidth}

Since the recursions in~(\ref{eq:G0ikrecgeqk}) are readily expressed as a backward and a 
forward recursion, we need to obtain the value of pivot element $G_{(0, k)}^k$ to solve 
them.
The following identities
readily follow by induction on the first and the second recursions in~(\ref{eq:G0ikrecgeqk}):
\begin{equation}
\label{eq:G0ikpivot}
\begin{split}
G_{(0, i)}^k & = \beta^{k-i} G_{(0, k)}^k, \quad i < k \\
G_{(0, k)}^k & = (1 + \cdots + (\beta q)^{l-1}) \beta (\lambda + p G_{(0, 0)}^k) + (\beta q)^l  G_{(0, k+l)}^k, \quad l \geqslant 1.
\end{split}
\end{equation}

From the first identity in~(\ref{eq:G0ikpivot}) we obtain $G_{(0, 0)}^k = \beta^{k} G_{(0, k)}^k$. 
Now, using that \linebreak  $(\beta q)^l G_{(0, k+l)}^k \to 0$ as $l \to \infty$ (because $G_{(b, i)}^\pi \leqslant 1/(1-\beta)$), we obtain
\begin{equation}
\label{eq:G0kkG00k}
G_{(0, k)}^k = \beta \frac{\lambda + p G_{(0, 0)}^k}{1-\beta q}.
\end{equation}

We thus have a $2 \times 2$ linear equation system characterizing  $G_{(0, 0)}^k$ and $G_{(0, k)}^k$, which yields
\[
G_{(0, k)}^k = \frac{\beta \lambda}{1 - \beta q - \beta^{k+1} p}.
\]

From this and~(\ref{eq:G0ikpivot}) we further obtain $G_{(0, k+l)}^k \equiv G_{(0, k)}^k$ for  $l \geqslant 1$,  completing the proof.
\end{proof}

\section[\appendixname~\thesection]{Cost Metric Evaluation}
\label{s:cmeval}
\begin{proof}[Proof of Lemma $\ref{lma:Feval}$]
It follows from~(\ref{eq:F0ik}) and~(\ref{eq:F1ik})
 that $F_{(1, i)}^k$ can be expressed in terms of the $F_{(0, j)}^k$ as 
\begin{equation}
\label{eq:F1ikF0ik}
F_{(1, i)}^k = 
\begin{cases}
F_{(0, i)}^k, & \quad i \leqslant k \\
\mu c_i + \mu F_{(0, 0)}^k + (1-\mu) F_{(0, i)}^k, &  \quad i > k.
\end{cases}
\end{equation}

On the other hand, we have the following recursions for $F_{(0, i)}^k$:
\begin{equation}
\label{eq:F0ikrecgeqk}
\begin{split}
F_{(0, i)}^k & = c_i + \beta F_{(0, i+1)}^k, \quad i < k \\
F_{(0, i)}^k & = c_i + \beta p c_{i+1} + \beta p F_{(0, 0)}^k + \beta q F_{(0, i+1)}^k, \enspace i \geqslant k.
\end{split}
\end{equation}

The first recursion follows from~(\ref{eq:F0ik}) and~(\ref{eq:F1ikF0ik}) as, for $i < k$, 
\vspace{-12pt}
\begin{adjustwidth}{-1cm}{0cm}
\begin{align*}
F_{(0, i)}^k & = c_i + \beta (1-\lambda) F_{(0, i+1)}^k + \beta \lambda F_{(1, i+1)}^k \\
& = c_i + \beta (1-\lambda) F_{(0, i+1)}^k + \beta \lambda F_{(0, i+1)}^k
  = c_i + \beta F_{(0, i+1)}^k.
\end{align*}
\end{adjustwidth}

The second recursion also follows from~(\ref{eq:F0ik}) and~(\ref{eq:F1ikF0ik}) because, for $i \geqslant k$,
\begin{align*}
F_{(0, i)}^k & = c_i + \beta (1-\lambda) F_{(0, i+1)}^k + \beta \lambda F_{(1, i+1)}^k \\
& = c_i + \beta (1-\lambda) F_{(0, i+1)}^k + \beta \lambda \big(\mu c_{i+1} + \mu F_{(0, 0)}^k + (1-\mu) F_{(0, i+1)}^k\big) \\
& = c_i + \beta p c_{i+1} + \beta p F_{(0, 0)}^k + \beta q F_{(0, i+1)}^k.
\end{align*}

Since the recursions in~(\ref{eq:F0ikrecgeqk}) are readily expressed as a backward 
recursion and a forward recursion, we need to obtain the value of pivot element $F_{(0, 
k)}^k$ to solve them.
For such a purpose we shall derive equations on $F_{(0, k)}^k$ and $F_{(0, 0)}^k$.
Thus, from the first recursion we readily obtain by induction the following identities:
\begin{equation}
\label{eq:F00kF0kk}
F_{(0, i)}^k = \sum_{j=i}^{k-1} \beta^{j-i} c_j + \beta^{k-i} F_{(0, k)}^k, \quad i < k.
\end{equation}

To analyze the second recursion in~(\ref{eq:F0ikrecgeqk}) we need to distinguish two cases, depending on whether $q = 0$ or $q > 0$.
Consider first the case $q = 0$, in which the second expression in~(\ref{eq:F0ikrecgeqk}) is not actually a recursion, as it reduces to
\begin{equation}
\label{eq:2ndrecq0}
F_{(0, i)}^k = c_i + \beta c_{i+1} + \beta F_{(0, 0)}^k, \enspace i \geqslant k.
\end{equation}

Now we solve the $2 \times 2$ linear equation system 
\begin{align*}
F_{(0, 0)}^k & = \sum_{j=0}^{k-1} \beta^{j} c_j + \beta^{k} F_{(0, k)}^k \\
F_{(0, k)}^k  & = c_k + \beta c_{k+1} + \beta F_{(0, 0)}^k
\end{align*}
 to obtain
\[
F_{(0, 0)}^k = \frac{1}{1 - \beta^{k+1}} \sum_{j=0}^{k+1} \beta^{j} c_j
\]
and
\begin{align*}
F_{(0, k)}^k & =  c_{k} + \beta c_{k+1} +  \frac{\beta}{1 - \beta^{k+1}} \sum_{j=0}^{k+1} \beta^{j} c_j = 
\frac{1}{1 - \beta^{k+1}} \Bigg(\sum_{j=0}^{k-1} \beta^{j+1} c_j +  c_{k} + \beta c_{k+1}\Bigg).
 \end{align*}
The above expressions  give the stated closed-form  formulae for $F_{(0, i)}^k$ in the case $q = 0$.

Consider now the case $q > 0$. Then, the following identities
readily follow by induction on the second recursion in~(\ref{eq:F0ikrecgeqk}): for $l \geqslant 1$,
\begin{equation}
\label{eq:F0ikpivot}
\begin{split}
F_{(0, k)}^k & =  \sum_{i=0}^{l-1} (\beta q)^i \big(c_{k+i} + \beta p c_{k+i+1} + \beta p F_{(0, 0)}^k\big) + (\beta q)^l F_{(0, k+l)}^k \\
 & =  \frac{\beta p (1- (\beta q)^l)}{ 1- \beta q} F_{(0, 0)}^k + (\beta q)^l F_{(0, k+l)}^k + c_{k} +  \beta \sum_{i=1}^{l-1} (\beta q)^{i-1} c_{k+i} + \beta p (\beta q)^{l-1} c_{k+l},
\end{split}
\end{equation}
where we have used that
\begin{align*}
\sum_{i=0}^{l-1} (\beta q)^{i} (c_{k+i} + \beta p c_{k+i+1}) & = \sum_{i=0}^{l-1} (\beta q)^{i} c_{k+i} + \beta p \sum_{i=0}^{l-1} (\beta q)^{i} c_{k+i+1}  \\
& = c_{k} + \beta q \sum_{i=1}^{l-1} (\beta q)^{i-1} c_{k+i} + \beta p \sum_{i=1}^{l} (\beta q)^{i-1} c_{k+i} \\
& = c_{k} +  \beta \sum_{i=1}^{l-1} (\beta q)^{i-1} c_{k+i} + \beta p (\beta q)^{l-1} c_{k+l}.
\end{align*}

Now, using that $(\beta q)^l F_{(0, k+l)}^k \to 0$ and $(\beta q)^{l-1} c_{k+l} \to 0$ as $l \to \infty$ (which follows from Assumption~\ref{ass:ci}), we obtain, letting $l \to \infty$,
\begin{equation}
\label{eq:F0kkF00k}
F_{(0, k)}^k =  \frac{\beta p}{1-\beta q} F_{(0, 0)}^k +  c_{k} +  \beta \sum_{i=1}^{\infty} (\beta q)^{i-1} c_{k+i}.
\end{equation}

We thus have a $2 \times 2$ linear equation system determining  $F_{(0, 0)}^k$ and $F_{(0, k)}^k$, which yields
\begin{align*}
F_{(0, k)}^k & = \frac{\beta p}{1-\beta q} \Bigg(\sum_{j=0}^{k-1} \beta^{j} c_j + \beta^{k} F_{(0, k)}^k\Bigg) +  c_{k} +  \beta \sum_{i=1}^{\infty} (\beta q)^{i-1} c_{k+i} \\
& = \frac{\beta^{k+1} p}{1-\beta q}  F_{(0, k)}^k + \frac{\beta p}{1-\beta q} \sum_{j=0}^{k-1} \beta^{j} c_j +   c_{k} +  \beta \sum_{i=1}^{\infty} (\beta q)^{i-1} c_{k+i},
\end{align*}
whence
\[
\bigg(1 - \frac{\beta^{k+1} p }{1-\beta q}\bigg) F_{(0, k)}^k = \frac{p}{1-\beta q} \sum_{j=0}^{k-1} \beta^{j+1} c_j +  c_{k} +  \beta \sum_{i=1}^{\infty} (\beta q)^{i-1} c_{k+i},
\]
i.e., 
\[
 \frac{1-\beta q- \beta^{k+1} p }{1-\beta q}F_{(0, k)}^k = \frac{p}{1-\beta q} \sum_{j=0}^{k-1} \beta^{j+1} c_j +  c_{k} +  \beta \sum_{i=1}^{\infty} (\beta q)^{i-1} c_{k+i}\]
i.e.,
\begin{align*}
 F_{(0, k)}^k = \frac{1}{1-\beta q- \beta^{k+1} p} \Bigg[p \sum_{j=0}^{k-1} \beta^{j+1} c_j  +  (1-\beta q) \bigg(c_{k}   +  \beta \sum_{i=1}^{\infty} (\beta q)^{i-1} c_{k+i}\bigg)\Bigg].
\end{align*}

This completes the proof.
\end{proof}

We next give the proof of Lemma $\ref{lma:mfeval}$.

\begin{proof}[Proof of Lemma $\ref{lma:mfeval}$]
First, we draw on the results in Section~\ref{s:cman} to represent $f_{(1, i)}^k$ as
\begin{equation}
\label{eq:f1ik}
\begin{split}
f_{(1, i)}^k & \triangleq F_{(1, i)}^{\langle 0, k\rangle} - F_{(1, i)}^{\langle 1, k\rangle} \\
& = F_{(0, i)}^k - \big(\mu c_i + \mu F_{(0, 0)}^k + (1-\mu) F_{(0, i)}^k\big) \\
& = 
\mu \big(F_{(0, i)}^k - F_{(0, 0)}^k - c_i\big) = \mu \big(\Phi_{i}^k - \Phi_{0}^k - c_i\big).
\end{split}
\end{equation}

In the case $q = 0$, we have
\begin{align*}
f_{(1, i)}^i & = \mu \big(\Phi_{i}^i - \Phi_{0}^i - c_i\big) =  \Phi_{i}^i - \sum_{j=0}^{i-1} \beta^{j} c_j - \beta^{i} \Phi_{i}^i - c_i  = 
\ (1 - \beta^{i}) \Phi_{i}^i - \sum_{j=0}^{i-1} \beta^{j} c_j - c_i \\
& = \frac{1 - \beta^{i}}{1- \beta^{i+1}} \bigg(\sum_{j=0}^{i-1} \beta^{j+1} c_j  +  c_{i}   +  \beta   c_{i+1}\bigg) - \sum_{j=0}^{i-1} \beta^{j} c_j - c_i \\
& =  \frac{\beta   (1 - \beta^{i})}{1- \beta^{i+1}}   c_{i+1} - \bigg(1 - \frac{\beta (1 - \beta^{i})}{1- \beta^{i+1}}\bigg) \sum_{j=0}^{i-1} \beta^{j} c_j - \bigg(1 -  \frac{1 - \beta^{i}}{1- \beta^{i+1}}\bigg) c_i  \\
& = \frac{1}{1- \beta^{i+1}} \bigg[\beta  (1 - \beta^{i})   c_{i+1} - (1 - \beta) \sum_{j=0}^{i} \beta^{j} c_j\bigg].
\end{align*}

As for the case $q > 0$, we have (recall that $\sigma_i \triangleq 1-\beta q- \beta^{i+1} p$)
\begin{align*}
f_{(1, i)}^i & = \mu \big(\Phi_{i}^i - \Phi_{0}^i - c_i\big) =  \mu \bigg[(1 - \beta^{i}) \Phi_{i}^i - \sum_{j=0}^{i-1} \beta^{j} c_j - c_i\bigg] \\
& = \mu \bigg[\frac{1 - \beta^{i}}{\sigma_i} \bigg(p \sum_{j=0}^{i-1} \beta^{j+1} c_j  +  (1-\beta q) \bigg(c_{i}   +  \beta \sum_{j=1}^{\infty} (\beta q)^{j-1} c_{i+j}\bigg)\bigg) - \sum_{j=0}^{i-1} \beta^{j} c_j - c_i\bigg] \\
& =  \mu \bigg[\frac{\beta (1 - \beta^{i})(1-\beta q)}{\sigma_i}  \sum_{j=1}^{\infty} (\beta q)^{j-1} c_{i+j}   \\
& \qquad - \bigg(1 - \frac{\beta (1 - \beta^{i}) p}{\sigma_i}\bigg) \sum_{j=0}^{i-1} \beta^{j} c_j  - 
\bigg(1 - \frac{(1 - \beta^{i})(1-\beta q)}{\sigma_i}\bigg) c_{i}\bigg]\\
& = \frac{\mu}{\sigma_i}\bigg[\beta (1 - \beta^{i})(1-\beta q)  \sum_{j=1}^{\infty} (\beta q)^{j-1} c_{i+j} - (1 - \beta) \sum_{j=0}^{i} \beta^{j} c_j\bigg].
\end{align*}

This completes the proof.
\end{proof}

%



\end{document}